\documentclass{amsart}
\usepackage{oldgerm}
\usepackage{latexsym}
\usepackage{amsmath}
\usepackage{amscd}
\usepackage{amssymb} 
\usepackage{amsmath}
\usepackage{amsthm}
\usepackage{latexsym}

\def\(#1;#2){\left\{\begin{array} {l} #1 \\ #2 \end{array}\right.}

\def\listtwo(#1;#2;#3;#4){\left\{\begin{array}{ll} #1 & #2 \\ #3 &#4
			       \end{array}\right.}
			       
\def\listthree(#1;#2;#3;#4;#5;#6){\left\{\begin{array}{ll} #1 & #2 \\ #3 &#4
			        \\ #5 & #6 \end{array}\right.}
			       
\def\smallmattwo(#1;#2;#3;#4){\left(\begin{smallmatrix} #1 & #2 \\ #3 &#4
			       \end{smallmatrix}\right)}

\def\ssmallmattwo(#1;#2;#3;#4){\left.\begin{smallmatrix} #1 & #2 \\ #3 &#4
			       \end{smallmatrix}\right.}
			       
\def\mattwo(#1;#2;#3;#4){\left(\begin{array}{cc} #1 & #2 \\ #3 &#4
			       \end{array}\right)}

\def\(#1;#2){\left\{\begin{array} {l} #1 \\ #2 \end{array}\right.}

\def\vecttwo(#1;#2){\left(\begin{array}{c} #1 \\ #2
\end{array}\right)} 

\def\matthree(#1;#2;#3;#4;#5;#6;#7;#8;#9){\left(\begin{array}{ccc} 
#1 & #2 & #3 \\
 #4 & #5 &#6 \\
 #7 & #8 &#9
\end{array}\right)}
			       
\def\smallmatthree(#1;#2;#3;#4;#5;#6;#7;#8;#9){\left(\begin{smallmatrix}{ccc} 
#1 & #2 & #3 \\
 #4 & #5 &#6 \\
 #7 & #8 &#9
\end{smallmatrix}\right)}
				       
\def\vecttwo(#1;#2){\left(\begin{array}{c} #1 \\ #2
\end{array}\right)} 

\def\vectthree(#1;#2;#3){\left(\begin{array}{c} #1 \\ #2 \\ #3
\end{array}\right)}

\theoremstyle{plain} 

\newtheorem{thm}{Theorem}[section] 
\newtheorem{lem}[thm]{Lemma}
\newtheorem{prop}[thm]{Proposition}

\newtheorem{thms}{Theorem}[subsection] 
\newtheorem{lems}[thms]{Lemma}
\newtheorem{props}[thms]{Proposition}

{Corollary}[section]
\newtheorem*{xcor}{Corollary} 

\theoremstyle{definition}

\theoremstyle{remark} 

\newtheorem*{acknow}{Acknowledgments} 

\begin{document}
  
\title [Koecher-Maa{\ss} series  of the Ikeda lift for $U(m,m)$]
{Koecher-Maa{\ss} series  of the Ikeda lift for $U(m,m)$ }

\author [H. Katsurada] {Hidenori KATSURADA \\ 
In  memory of Professor Hiroshi Saito}

\date{}

\maketitle

\begin{abstract}
Let $K={\bf Q}(\sqrt{-D})$ be an imaginary quadratic field with discriminant $-D,$  and $\chi$ the Dirichlet character corresponding to the extension $K/{\bf Q}.$ 
Let $m=2n$ or $2n+1$ with $n$ a positive integer. Let $f$ be a primitive form of weight $2k+1$ and character $\chi$ for $\varGamma_0(D),$ 
  or a primitive form of weight $2k$ for $SL_2({\bf Z})$ according as $m=2n,$ or $m=2n+1.$ For such an $f$ let $I_m(f)$ be the lift of $f$ to the space of Hermitian modular forms constructed by Ikeda.
  We then give an explicit formula of the Koecher-Maass series $L(s,I_m(f))$ of $I_m(f).$ 
  This is a generalization of Mizuno \cite{Mi}.

\end{abstract}

\footnote[0]{2010 {\it{Mathematics Subject Classification}} Primary 11F55, 11F67,
.}

\section{Introduction} 
In \cite{Mi}, Mizuno gave  explicit formulas  of the Koecher-Maass series  of the Hermitian Eisenstein series of degree two and of the Hermitian Maass lift. In this paper, we give an explicit formula of the Koecher-Maass series of the Hermitian Ikeda lift.  
Let $K={\bf Q}(\sqrt{-D})$ be an imaginary quadratic field with discriminant $-D.$ Let $\mathcal O$ be the ring of integers in $K,$ and $\chi$ the Kronecker character corresponding to the extension $K/{\bf Q}.$ For a non-degenerate Hermitian matrix or alternating matrix $T$ with entries in $K,$ let ${\mathcal U}_T$ be the unitary group defined over ${\bf Q},$ whose group ${\mathcal U}_T(R)$ of $R$-valued points is given by
$${\mathcal U}_T(R)=\{ g \in GL_{m}(R \otimes K) \ | \ {}^t\overline{g}Tg = T \}$$
for any ${\bf Q}$-algebra $R,$ where $\overline {g}$ denotes the automorphism of $M_n(R \otimes K)$ induced by the non-trivial automorphism of $K$ over ${\bf Q}.$ 
We also define the special unitary group $\mathcal {SU}_T$ over ${\bf Q}_p$ by $\mathcal {SU}_T={\mathcal U}_T \cap R_{K/{\bf Q}}(SL_m),$ where $R_{K/{\bf Q}}$ is the Weil restriction. In particular we write ${\mathcal U}_T$ as ${\mathcal U}^{(m)}$ or $U(m,m)$ if $T=\smallmattwo(O;-1_m;1_m;O).$ For a more precise description of ${\mathcal U}^{(m)}$ see Section 2. Put $\varGamma_K^{(m)}=U(m,m)({\bf Q}) \cap GL_{2m}(\mathcal O).$ 
For a modular form $F$ of weight $2l$ and character $\psi$ for $\varGamma_K^{(m)}$ we define the Koecher-Maass series
$L(s,F)$ of $F$ by
$$L(s,F)=\sum_{T} {c_F(T) \over e^*(T) (\det T)^s},$$
where $T$ runs over all $SL_m({\mathcal O})$-equivalence classes of positive definite semi-integral Hermitian matrices of 
degree $m,$ $c_F(T)$ denotes the $T$-th Fourier coefficient of $F,$ and $e^*(T)=\#(\mathcal {SU}_T({\bf Q}) \cap SL_m({\mathcal O})).$

Let $k$ be a non-negative integer. Then for a primitive form $f \in {\textfrak S}_{2k+1}(\varGamma_0(D),\chi)$  Ikeda \cite{Ik3} constructed a lift $I_{2n}(f)$ of $f$ to the space of modular forms of weight $2k+2n$ and a character $\det^{-k-n}$ for $\varGamma_K^{(2n)}.$ This is a generalization of the Maass lift considered by Kojima \cite{Koj}, Gritsenko \cite{G}, Krieg \cite{Kri} and Sugano \cite{Su}. Similarly for a primitive form $f \in {\textfrak S}_{2k}(SL_2({\bf Z}))$  he  constructed a lift $I_{2n+1}(f)$ of $f$ to the space of modular forms of weight $2k+2n$ and a character $\det^{-k-n}$ for $\varGamma_K^{(2n+1)}.$ For the rest of this section, let $m=2n$ or $m=2n+1.$ We then call $I_{m}(f)$  the Ikeda lift of $f$ for $U(m,m)$ or the Hermitian Ikeda lift of degree $m.$ Ikeda also showed that the automorphic form $Lift^{(m)}(f)$  on the adele group ${\mathcal U}^{(m)}({\bf A})$ associated with $I_{m}(f)$ is a cuspidal Hecke eigenform whose standard $L$-function coincides with  
$$\prod_{i=1}^{m} L(s+k+n-i+1/2,f)L(s+k+n-i+1/2,f,\chi),$$ where $L(s+k+n-i+1/2,f)$ is the Hecke $L$-function of $f$ and $L(s+k+n-i+1/2,f,\chi)$ is its "modified  twist" by $\chi.$ For the precise definition of $L(s+k+n-i+1/2,f,\chi)$ see Section 2. We also call $Lift^{(m)}(f)$ the adelic Ikeda lift of $f$ for $U(m,m).$ Then we express the Kocher-Maass series of $I_m(f)$ in terms of the $L$-functions related to $f.$ 
 This  result was already obtained in the case $m=2$ by Mizuno \cite{Mi}.  

 The method we use is similar to that in the proof of the main result of \cite{I-K2} or \cite{I-K3}. We explain it more precisely. In Section 3, we reduce our computation to  a computation of certain formal power series $\hat   P_{m,p}(d;X,t)$ in  $t$ associated with local Siegel series similarly to \cite{I-K2}  (cf. Theorem 3.4 and Section 5). 

Section 4 is devoted to the computation of them. This computation is similar to that in \cite{I-K2}, but we should be careful in dealing with the case where $p$ is ramified in $K$. After such an elaborate computation, we can get explicit formulas of $\hat P_{m,p}(d;X,t)$ for all prime numbers $p$ (cf. Theorems 4.3.1, 4.3.2, and 4.3.6). In Section 5, by using explicit formulas for $\hat P_{m,p}(d;X,t)$, we immediately get an explicit formula of $L(s,I_m(f)).$ 

Using the same argument as in the proof our main result,  we can give an explicit formula of the Koecher-Maass series of the Hermitian Eisenstein series of any degree, which can be regarded as a zeta function of a certain prehomogeneous vector space. We also note that  the method used in this paper is useful for giving an explicit formula for the Rankin-Selberg series of the Hermitian Ikeda lift, and as a result we can prove  the period relation of the Hermitian Ikeda lift,  which was conjectured by Ikeda \cite {Ik3}. We will discuss these topics in  subsequent papers \cite{Kat4} and \cite{Kat5}.

\begin{acknow}
The author thanks Professor T. Watanabe and Professor R. Schulze-Pillot for giving him many crucial comments on  the mass formula for the unitary group.
He also thanks Professor Y. Mizuno, Professor T. Ikeda, and A. Yukie for useful discussions.
The author was partly supported by JSPS KAKENHI Grant Number 24540005.  \end{acknow}

\bigskip
  
\qquad {\bf Notation.}  Let $R$ be a commutative ring. We denote by $R^{\times}$ and $R^*$  the semigroup of non-zero elements of $R$ and the unit group of $R,$  respectively. For a subset $S$ of $R$ we denote by $M_{mn}(S)$ the set of
$(m,n)$-matrices with entries in $S.$ In particular put $M_n(S)=M_{nn}(S).$ 
  Put $GL_m(R) = \{A \in M_m(R) \ | \ \det A \in R^* \},$ where $\det
A$ denotes the determinant of a square matrix $A$. Let $K_0$ be a field, and  $K$ a quadratic extension of $K_0,$ or $K=K_0 \oplus K_0.$ In the latter case, we regard $K_0$ as a subring of $K$ via the diagonal embedding.  We also identify $M_{mn}(K)$ with $M_{mn}(K_0) \oplus M_{mn}(K_0)$ in this case.
If $K$ is a quadratic extension of $K_0,$ let  $\rho$ be the non-trivial automorphism of $K$ over $K_0,$ and if $K=K_0 \oplus K_0,$ let $\rho$ be the automorphism of $K$ defined by $\rho(a,b)=(b,a)$ for $(a,b) \in K.$ We sometimes write $\overline {x}$ instead of $\rho(x)$ for $x \in K$ in both cases. Let $R$ be a subring of $K.$ For an $(m,n)$-matrix $X=(x_{ij})_{m \times n}$ write $X^*=(\overline{x_{ji}})_{n \times m},$ and for an $(m,m)$-matrix
$A$, we write $A[X] = X^* A X.$ Let ${\rm Her}_n(R)$ denote
the set of Hermitian  matrices of degree $n$ with entries in
$R$, that is the subset of $M_n(R)$ consisting of matrices $X$ such that $X^*=X.$   Then a Hermitian matrix $A$ of degree $n$ with entries in $K$ is said to be semi-integral over $R$  if ${\rm tr}(AB) \in K_0 \cap R$ for any $B \in {\rm Her}_n(R),$ where ${\rm tr}$ denotes the trace of a matrix. We denote by $\widehat{{\rm Her}}_n(R)$ the set of semi-integral matrices of degree $n$ over $R.$ 

  For a subset $S$ of
$M_n(R)$ we denote by $S^{\times}$ the subset of $S$
consisting of non-degenerate matrices. If $S$ is a subset of ${\rm Her}_n({\bf C})$ with ${\bf C}$ the field of complex  numbers, we denote by $S^+$ the subset of $S$ consisting of positive definite matrices. The group 
  $GL_n(R)$ acts on the set ${\rm Her}_n(R)$ in the following way:
  $$ GL_n(R) \times {\rm Her}_n(R) \ni (g,A) \longrightarrow g^* Ag \in {\rm Her}_n(R).$$
  Let $G$ be a subgroup of $GL_n(R).$ For a $G$-stable subset ${\mathcal B}$ of ${\rm Her}_n(R)$ we denote by ${\mathcal B}/G$ the set of equivalence classes of ${\mathcal B}$ under the action of $G.$ We sometimes identify ${\mathcal B}/G$ with a complete set of representatives of ${\mathcal B}/G.$ We abbreviate ${\mathcal B}/GL_n(R)$ as ${\mathcal B}/\sim$ if there is no fear of confusion. Two Hermitian matrices $A$ and $A'$ with
entries in $R$ are said to be $G$-equivalent and write $A \sim_{G} A'$ if there is
an element $X$ of $G$ such that $A'=A[X].$ 
For square matrices $X$ and $Y$ we write $X \bot Y = \mattwo(X;O;O;Y).$

We put ${\bf e}(x)=\exp(2 \pi \sqrt{-1} x)$ for $x \in {\bf C},$ and for a prime number $p$ we denote by ${\bf e}_p(*)$ the continuous additive character of ${\bf Q}_p$ such that ${\bf e}_p(x)= {\bf e}(x)$ for $x \in {\bf Z}[p^{-1}].$ 

For a prime number $p$ we denote by ${\rm ord}_p(*)$ the additive valuation of ${\bf Q}_p$ normalized so that ${\rm ord}_p(p)=1,$ and put  $|x|_p=p^{-{\rm ord}_p(x)}.$ Moreover we denote by $|x|_{\infty}$ the absolute value of $x \in {\bf C}.$
Let $K$ be an imaginary quadratic field, and $\mathcal O$ the ring of integers in $K.$  For a prime number $p$ put $K_p=K \otimes {\bf Q}_p,$ and ${\mathcal O}_p={\mathcal O} \otimes {\bf Z}_p.$ Then $K_p$ is a quadratic extension of ${\bf Q}_p$ or $K_p \cong {\bf Q}_p \oplus {\bf Q}_p.$  In the former case, for  $x \in K_p,$ we denote by $\overline {x}$ the conjugate of $x$ over ${\bf Q}_p.$ In the latter case, we identify $K_p$ with ${\bf Q}_p \oplus {\bf Q}_p,$ and for $x=(x_1,x_2)$ with $x_i \in {\bf Q}_p,$ we put $\overline {x}=(x_2,x_1).$ For $x \in K_p$ we define the norm $N_{K_p/{\bf Q}_p}(x)$ by $N_{K_p/{\bf Q}_p}(x)=x\overline {x},$ and 
 put $\nu_{K_p}(x)={\rm ord}_p(N_{K_p/{\bf Q}_p}(x)),$ and $|x|_{K_p}=|N_{K_p/{\bf Q}_p}(x)|_p.$ Moreover put $|x|_{K_{\infty}}=|x \overline{x}|_{\infty}$ for $x \in {\bf C}.$ 

\section{Main results }
For a positive integer $N$ let 
$$\varGamma_0(N)=\{ \mattwo(a;b;c;d) \in SL_2({\bf Z}) \ | \ c \equiv 0 \ {\rm mod} \ N  \},$$
and for a Dirichlet character $\psi$ mod $N,$ we denote by ${\textfrak M}_{l}(\varGamma_0(N),\psi)$ the space of modular forms of weight $l$ for 
$\varGamma_0(N)$ and nebentype $\psi,$ and by ${\textfrak S}_{l}(\varGamma_0(N),\psi)$ its subspace consisting of cusp forms. We simply write
 ${\textfrak M}_{l}(\varGamma_0(N),\psi)$ (resp. ${\textfrak S}_{l}(\varGamma_0(N),\psi)$)  as ${\textfrak M}_{l}(\varGamma_0(N))$ (resp. as ${\textfrak S}_{l}(\varGamma_0(N))$) if $\psi$ is the trivial character.   

Throughout the paper, we fix an imaginary quadratic extension $K$ of ${\bf Q}$ with discriminant $-D,$ and denote by ${\mathcal O}$ the ring of integers in $K.$  For such a $K$ let ${\mathcal U}^{(m)}=U(m,m)$ be the unitary group defined in Section 1. Put $J_m=\mattwo(O_m;-1_m;1_m;O_m),$ where $1_m$ denotes the unit matrix of degree $m.$ Then   
$${\mathcal U}^{(m)}({\bf Q})=\{M \in GL_{2m}(K)  \ | \  J_m[M]= J_m  \}.$$
 Put 
$$\varGamma^{(m)}=\varGamma_K^{(m)}={\mathcal U}^{(m)}({\bf Q}) \cap GL_{2m}({\mathcal O}). $$
  Let ${\textfrak H}_m$ be the Hermitian upper half-space defined by 
$${\textfrak H}_m=\{Z \in M_m({\bf C}) \ | \ {1 \over 2\sqrt{-1}} (Z-Z^*) \ {\rm is \ positive \ definite} \}.$$
The group ${\mathcal U}^{(m)}({\bf R})$ acts on ${\textfrak H}_m$ by
$$ g \langle Z \rangle =(AZ+B)(CZ+D)^{-1} \ {\rm for} \ g=\smallmattwo(A;B;C;D) \in {\mathcal U}^{(m)}({\bf R}), Z \in {\textfrak H}_m.$$
We also put $j(g,Z)=\det (CZ+D)$ for such $Z$ and $g.$
  Let $l$ be an integer. For a subgroup $\varGamma$ of ${\mathcal U}^{(m)}({\bf Q})$ commensurable with $\varGamma^{(m)}$ and a character $\psi$ of $\varGamma,$ we denote by ${\textfrak M}_{l}(\varGamma,\psi)$  the space of holomorphic modular forms of weight $l$  with character $\psi$ for $\varGamma.$   We denote by ${\textfrak S}_{l}(\varGamma,\psi)$ the subspace of ${\textfrak M}_{l}(\varGamma,\psi)$ consisting of cusp forms. In particular, if $\psi$ is the character of $\varGamma$ defined  by $\psi(\gamma)=(\det \gamma)^{-l}$ for $\gamma \in \varGamma,$ we write ${\textfrak M}_{2l}(\varGamma,\psi)$ as  ${\textfrak M}_{2l}(\varGamma,\det^{-l}),$ and so on.  
  Let $F(z)$ be an element of ${\textfrak M}_{2l}(\varGamma^{(m)},\det^{-l}).$ We then define the Koecher-Maass series $L(s,F)$ for $F$ by
  $$L(s,F)=\sum_{T \in {\widehat {\rm Her}_m({\mathcal O})}^+/SL_n(\mathcal O)} {c_{F}(T)  \over (\det T)^{s} e^*(T)},$$
 where  $c_F(T)$ denotes the $T$-th Fourier coefficient of $F,$ and $e^*(T)=\#(\mathcal {SU}_T({\bf Q}) \cap SL_m(\mathcal O)).$   
 
Now we consider the adelic modular form.  
Let ${\bf A}$ be the adele ring of ${\bf Q},$ and  ${\bf A}_{f}$ the non-archimedian factor of ${\bf A}.$ 
  Let $h=h_K$ be a class number of $K.$ 
Let $G^{(m)}={\rm Res}_{K/{\bf Q}}(GL_m),$ and $G^{(m)}({\bf A})$ be the adelization of $G^{(m)}.$ Moreover  put ${\mathcal C}^{(m)}=\prod_p GL_m({\mathcal O}_p).$ Let ${\mathcal U}^{(m)}({\bf A})$ be the adelization of ${\mathcal U}^{(m)}.$ We define the  compact subgroup $\mathcal K_0^{(m)}$ of ${\mathcal U}^{(m)}({\bf A}_f)$ by ${\mathcal U}^{(m)}({\bf A}) \cap \prod_p GL_{2m}({\mathcal O}_p),$ where $p$ runs over all rational primes. 
 Then we have
$${\mathcal U}^{(m)}({\bf A})=\bigsqcup_{i=1}^h {\mathcal U}^{(m)}({\bf Q}) \gamma_i \mathcal K_0 ^{(m)}{\mathcal U}^{(m)}({\bf R})$$
with some subset $\{\gamma_1,...,\gamma_h \}$ of ${\mathcal U}^{(m)}({\bf A}_f).$ We can take $\gamma_i $ as
$$\gamma_i =\mattwo(t_i;0;0;t_i^{* -1}),$$ 
where $\{ t_i \}_{i=1}^h =\{(t_{i,p}) \}_{i=1}^h $ is a certain subset of   $G^{(m)}({\bf A}_f)$ such that $t_1=1,$ and
 $$G^{(m)}({\bf A})=\bigsqcup_{i=1}^h G^{(m)}({\bf Q}) t_i G^{(m)}({\bf R}){\mathcal C}^{(m)}.$$
  Put $\varGamma_i={\mathcal U}^{(m)}({\bf Q}) \cap \gamma_i{\mathcal K}_0 \gamma_i^{-1}{\mathcal U}^{(m)}({\bf R}).$ Then for an element $(F_1,...,F_h) \in \bigoplus_{i=1}^h  {\textfrak M}_{2l}(\varGamma_i,\det^{-l}),$ we define $(F_1,...,F_h)^\sharp$ by
$$(F_1,...,F_h)^{\sharp}(g)=F_i(x\langle {\bf i} \rangle)j(x,{\bf i})^{-2l} (\det x)^l$$
for $g=u\gamma_ix \kappa$ with $u \in {\mathcal U}^{(m)}({\bf Q}),x \in {\mathcal U}^{(m)}({\bf R}),\kappa \in {\mathcal K}_0.$  We denote by ${\mathcal M}_{2l}({\mathcal U}^{(m)}({\bf Q}) \backslash {\mathcal U}^{(m)}({\bf A}), \det^{-l})$ the space of automorphic forms obtained in this way. We also put
$${\mathcal S}_{2l}({\mathcal U}^{(m)}({\bf Q}) \backslash {\mathcal U}^{(m)}({\bf A}),{\det}^{-l})=\{ (F_1,...,F_h)^{\sharp} \ | \ F_i \in {\textfrak S}_{2l}(\varGamma_i,{\det}^{-l}) \}.$$
We can define the Hecke operators which act on the space \\
${\mathcal M}_{2l}({\mathcal U}^{(m)}({\bf Q}) \backslash {\mathcal U}^{(m)}({\bf A}), \det^{-l}).$ For the precise definition of them, see \cite{Ik3}. 

Let $\widehat {\rm Her}_m({\mathcal O})$ be the set of semi-integral Hermitian matrices over ${\mathcal O}$ of degree $m$ as in the Notation. We note that $A$ belongs to  $\widehat {\rm Her}_m({\mathcal O})$ if and only if its diagonal components are rational integers and $\sqrt{-D} A \in {\rm Her}_m({\mathcal O}).$   For a non-degenerate Hermitian  matrix $B$ with entries in  $K_p$ of degree $m,$ put
 $\gamma(B)=(-D)^{[m/2]}\det B.$ 

 Let  
$\widehat {\rm Her}_m({\mathcal O}_p)$ be the set of semi-integral matrices over ${\mathcal O}_p$ of degree $m$ as in the Notation.@We put $\xi_p=1,-1,$ or
 $0$ according as $K_p = {\bf Q}_p \oplus {\bf Q}_p, K_p$ is
 an unramified quadratic extension of ${\bf Q}_p,$ or $K_p$
 is a  ramified quadratic extension of ${\bf Q}_p.$ 
For $T \in \widehat {\rm Her}_{m}({\mathcal O}_p)^{\times}$ we define the local Siegel series $b_p(T,s)$ by   $$b_p(T,s)=\sum_{R \in {\rm Her}_n(K_p)/{\rm Her}_n({\mathcal O}_p)} {\bf e}_p({\rm tr}(TR))p^{-{\rm ord}_p(\mu_p(R))s},$$ where $\mu_p(R)=[R{\mathcal O}_p^m+{\mathcal O}_p^m:{\mathcal O}_p^m].$  We remark that there exists a unique polynomial 
 $F_p(T,X)$ in $X$ such that 
 $$b_p(T,s)=F_p(T,p^{-s})\prod_{i=0}^{[(m-1)/2]}(1-p^{2i-s})\prod_{i=1}^{[m/2]} (1-\xi_p p^{2i-1-s}) $$ 
(cf. Shimura \cite{Sh1}). 
We then define a Laurent polynomial $\widetilde F_p(T,X)$ as
$$\widetilde F_p(T,X)=X^{-{\rm ord}_p(\gamma(T))}F_p(T,p^{-m}X^{2}).$$
We remark that we have 
$$\widetilde F_p(T,X^{-1})=(-D,\gamma(T))_p \widetilde F_p(T,X)  \qquad {\rm if } \ m \ {\rm is \ even},$$ 
$$\widetilde F_p(T,\xi_pX^{-1})=\widetilde F_p(T,X)  \qquad {\rm if } \ m \ {\rm is \ even \ and} \ p \nmid D,$$ 
and
$$\widetilde F_p(T,X^{-1})=\widetilde F_p(T,X)  \qquad {\rm if } \ m \ {\rm is \ odd}$$ (cf. \cite{Ik3}). Here $(a,b)_p$ is the Hilbert symbol of $a,b \in {\bf Q}_p^{\times}.$ Hence we have
$$\widetilde F_p(T,X)=(-D,\gamma(B))_p^{m-1} X^{{\rm ord}_p(\gamma(T))}F_p(T,p^{-m}X^{-2}).$$
Now we put
 $${\widehat {\rm Her}}_{m}({\mathcal O})_i^+=\{T \in {\rm Her}_m(K)^+ \ | \  t_{i,p}^*T t_{i,p} \in \widehat {\rm Her}_m({\mathcal O}_p) \ {\rm for \ any } \ p \}.$$

First let $k$ be a non-negative integer, and  $m=2n$  a positive even integer. Let 
$$f(z)=\sum_{N=1}^{\infty}a(N){\bf e}(Nz)$$
 be a  primitive form in ${\textfrak S}_{2k+1}(\varGamma_0(D),\chi).$  For a prime number $p$ not dividing $D$ let $\alpha_p \in {\bf C}$ such that $\alpha_p+\chi(p)\alpha_p^{-1}=p^{-k}a(p),$ and for $p \mid D$ put $\alpha_p=p^{-k}a(p).$  We note that $\alpha_p \not=0$ even if $p | D.$ Then for the Kronecker character $\chi$ we  define Hecke's $L$-function $L(s,f,\chi^i)$ twisted by $\chi^i$ as 
  $$L(s,f,\chi^i)=\prod_{p \nmid D}\{(1-\alpha_p p^{-s+k}\chi(p)^i)(1-\alpha_p^{-1} p^{-s+k} \chi(p)^{i+1})\}^{-1}$$
$$\times 
\left\{\begin{array}{ll}
\prod_{p \mid D} (1-\alpha_p p^{-s+k})^{-1} & {\rm if} \ i \ {\rm is \ even} \\
\prod_{p \mid D} (1-\alpha_p^{-1} p^{-s+k})^{-1} & {\rm if} \ i \ {\rm is \ odd}. \end{array} 
\right. $$
 In particular, if $i$ is even,  we sometimes write $L(s,f,\chi^i)$ as $L(s,f)$ as usual. Moreover for $i=1,...,h$ we define a Fourier series
 $$I_m(f)_i(Z)= \sum_{T \in \widehat {\rm Her}_{m}({\mathcal O})_i^+}a_{I_{m}(f)_i}(T){\bf e}({\rm tr}(TZ)),$$
 where
 $$a_{I_{2n}(f)_i}(T)= |\gamma (T)|^{k} \prod_p |\det (t_{i,p}) \det(\overline{t_{i,p}})|_p^{n}\widetilde F_p(t_{i,p}^*T t_{i,p},\alpha_p^{-1}).$$
Next let $k$ be a positive integer and $m=2n+1$  a positive odd integer. Let 
$$f(z)=\sum_{N=1}^{\infty}a(N){\bf e}(Nz)$$
 be a  primitive form in ${\textfrak S}_{2k}(SL_2({\bf Z})).$  For a prime number $p$ let $\alpha_p \in {\bf C}$ such that $\alpha_p+\alpha_p^{-1}=p^{-k+1/2}a(p).$ Then we  define Hecke's $L$-function $L(s,f,\chi^i)$ twisted by $\chi^i$ as   $$L(s,f,\chi^i)$$
  $$=\prod_{p }\{(1-\alpha_p p^{-s+k-1/2}\chi(p)^{i})(1-\alpha_p^{-1} p^{-s+k-1/2} \chi(p)^{i})\}^{-1}.$$
 In particular, if $i$ is even we write $L(s,f,\chi^i)$ as $L(s,f)$ as usual. Moreover  for $i=1,...,h$ we define a Fourier series
 $$I_{2n+1}(f)_i(Z)= \sum_{T \in \widehat {\rm Her}_{2n+1}({\mathcal O})_i^+} a_{I_{2n+1}(f)_i}(T){\bf e}({\rm tr}(TZ)),$$
 where
 $$a_{I_{2n+1}(f)_i}(T)= {|\gamma(T)|}^{k-1/2} \prod_p |\det(t_{i,p}) \det(\overline{t_{i,p}})|_p^{n+1/2}\widetilde F_p(t_{i,p}^*T t_{i,p},\alpha_p^{-1}).$$
 
 {\bf Remark.}  In \cite{Ik3}, Ikeda defined $\widetilde F_p(T,X)$ as 
$$\widetilde F_p(T,X)=X^{{\rm ord}_p(\gamma(T))}F_p(T,p^{-m}X^{-2}),$$
and we define it by replacing $X$ with $X^{-1}$ in this paper. This change does not affect the results.

 Then Ikeda \cite{Ik3} showed  the following: 
 
 \bigskip

\begin{thm}
Let $m=2n$ or $2n+1.$ Let  $f$ be a primitive form in 
${\textfrak S}_{2k+1}(\varGamma_0(D),\chi)$ or in ${\textfrak S}_{2k}(SL_2({\bf Z}))$ according as $m=2n$ or $m=2n+1.$ Moreover let $\varGamma_i$ be the subgroup of ${\mathcal U}^{(m)}$ defined as above. Then $I_m(f)_i(Z)$ is an element of ${\textfrak S}_{2k+2n}(\varGamma_i,\det^{-k-n})$ for any $i.$ In particular, $I_m(f):=I_m(f)_1$ is an element of ${\textfrak S}_{2k+2n}(\varGamma^{(m)},\det^{-k-n}).$
 \end{thm}

 \bigskip
 This is a Hermitian analogue of the lifting constructed in \cite{Ik1}. We call $I_m(f)$  the Ikeda lift of $f$ for ${\mathcal U}^{(m)}.$
 
 It follows from Theorem 2.1 that  we can define an element $(I_m(f)_1,...,I_m(f)_h)^{\sharp}$ of ${\mathcal S}_{2k+2n}({\mathcal U}^{(m)}({\bf Q}) \backslash {\mathcal U}^{(m)}({\bf A}),\det^{-k-n}),$ which we write $Lift^{(m)}(f).$ 
  
 \bigskip
 
\begin{thm}
 Let $m=2n$ or $2n+1.$ Suppose that $Lift^{(m)}(f)$ is not identically zero. Then $Lift^{(m)}(f)$ is a Hecke eigenform in ${\mathcal S}_{2k+2n}({\mathcal U}^{(m)}({\bf Q}) \backslash {\mathcal U}^{(m)}({\bf A}) ,\det^{-k-n})$ and its standard $L$-function $L(s,Lift^{(m)}(f),{\rm st})$ coincides with  
  $$\prod_{i=1}^m L(s+k+n-i+1/2,f)L(s+k+n-i+1/2,f,\chi)$$
  up to bad Euler factors.
 \end{thm}

\bigskip

We call  $Lift^{(m)}(f)$ the adelic Ikeda lift of $f$ for ${\mathcal U}^{(m)}.$


 Let
 $Q_D$ be the set of prime divisors of $D.$ For each prime $q \in Q_D,$ put $D_q=q^{{\rm ord}_q(D)}.$ We define a Dirichlet character $\chi_q$ by
 $$\chi_q(a)= \listtwo(\chi(a'); \ {\rm if} \ (a,q)=1 ; 0 ; \ {\rm if} \ q|a) ,$$
 where $a'$ is an integer such that
 $$a' \equiv a \ {\rm mod} \ D_q \quad \ {\rm and} \ a' \equiv 1 \ {\rm mod} \ DD_q^{-1}.$$
  For a subset $Q$ of $Q_D$ put
  $\chi_Q=\prod_{q \in Q} \chi_q$ and $\chi'_Q=\prod_{q \in Q_D, q \not\in Q} \chi_q.$ Here we make the convention that $\chi_Q=1$ and $\chi'_Q=\chi$ if $Q$ is the empty set. Let 
  $$f(z)=\sum_{N=1}^{\infty} c_f(N){\bf e}(Nz)$$
   be a primitive form in ${\textfrak S}_{2k+1}(\varGamma_0(D),\chi).$ Then there exists a primitive form 
   $$f_Q(z)=\sum_{N=1}^{\infty} c_{f_Q}(N){\bf e}(Nz)$$
   such that 
   $$c_{f_Q}(p)=\chi_Q(p)c_f(p) \ {\rm for} \ p \not\in Q$$
   and
   $$c_{f_Q}(p)=\chi'_Q(p)\overline{c_f(p)} \ {\rm for} \ p \in Q.$$
 Let $L(s,\chi^i)=\zeta(s)$ or $L(s,\chi)$ according as $i$ is even or odd, where $\zeta(s)$ and $L(s,\chi)$ are Rimann's zeta function, and the Dirichlet $L$-function for $\chi$, respectively. Moreover we define $\widetilde {\Lambda}(s,\chi^i)$ by
$$\widetilde {\Lambda}(s,\chi^i)=2(2\pi )^{-s}\Gamma(s)L(s,\chi^i)$$
with $\Gamma(s)$ the Gamma function.

 Then our main results in this paper are as follows:
 
 \bigskip

\begin{thm}
Let $k$ be a nonnegative integer and $n$ a positive integer. Let $f$ be a primitive form in ${\textfrak S}_{2k+1}(\varGamma_0(D),\chi).$ Then, we  have
$$L(s,I_{2n}(f))=D^{ns+n^2-n/2-1/2}2^{-2n+1} $$
$$ \times \prod_{i=2}^{2n} \widetilde {\Lambda}(i,\chi^i) \sum_{Q \subset Q_D } \chi_Q((-1)^n)\prod_{j=1}^{2n}L(s-2n+j,f_Q,\chi^{j-1}) .$$

\end{thm}

\begin{thm}
Let $k$ be a positive integer and $n$ a non-negative integer. Let  $f$ be  a primitive form in ${\textfrak S}_{2k}(SL_2({\bf Z})).$ Then, we  have
$$L(s,I_{2n+1}(f))=D^{ns+n^2+3n/2}2^{-2n} \prod_{i=2}^{2n+1} \widetilde {\Lambda}(i,\chi^i)\prod_{j=1}^{2n+1}L(s-2n-1+j,f,\chi^{j-1}).$$
\end{thm}

\bigskip
{\bf Remark.}  We note that $L(s,I_{2n+1}(f))$ has an Euler product.

\section{Reduction to local computations}
To prove our main result, we reduce the problem to local computations. 
Let $K_p=K \otimes {\bf Q}_p$ and ${\mathcal O}_p={\mathcal O} \otimes {\bf Z}_p$ as in  Notation. Then $K_p$ is  a quadratic extension of ${\bf Q}_p$ or $K_p={\bf Q}_p \oplus {\bf Q}_p.$ 
In the former case let  $f_p$ the exponent of the conductor of $K_p/{\bf Q}_p.$ If $K_p$ is ramified over ${\bf Q}_p,$ put $e_p=f_p -\delta_{2,p},$ where $\delta_{2,p}$ is Kronecker's delta. 
If  $K_p$ is unramified over ${\bf Q}_p,$ put $e_p=f_p=0.$ 
In the latter case, put  $e_p=f_p=0.$  
Let $K_p$ be a quadratic extension of ${\bf Q}_p,$ and $\varpi=\varpi_p$ and $ \pi=\pi_p$ be  prime elements of $K_p$ and ${\bf Q}_p$, respectively.   If $K_p$ is unramified over ${\bf Q}_p,$ we take $\varpi=\pi=p.$
 If $K_p$ is ramified over ${\bf Q}_p,$ we take $\pi$ so that $\pi=N_{K_p/{\bf Q}_p}(\varpi).$
Let $K_p={\bf Q}_p \oplus {\bf Q}_p.$ Then put $\varpi=\pi=p.$ 
Let $\chi_{K_p}$ be the quadratic character of ${\bf Q}_p^{\times}$ corresponding to the quadratic extension $K_p/{\bf Q}_p.$ We note 
that we have $\chi_{K_p}(a)=(-D_0,a)_p$ for $a \in {\bf Q}_p^{\times}$ if $K_p={\bf Q}_p(\sqrt{-D_0})$ with $D_0 \in {\bf Z}_p.$
Moreover put  $\widetilde {\rm Her}_{m}({\mathcal O}_p)=p^{e_p}\widehat {\rm Her}_{m}({\mathcal O}_p).$ We note that $\widetilde {\rm Her}_{m}({\mathcal O}_p)={\rm Her}_{m}({\mathcal O}_p)$ if $K_p$ is not ramified  over ${\bf Q}_p.$ Let $K$ be an imaginary quadratic extension of ${\bf Q}$ with  discriminant $-D.$ We then put $\widetilde D=\prod_{p | D} p^{e_p},$ and
$\widetilde {{\rm Her}}_{m}({\mathcal O})=\widetilde D {\rm Her}_{m}({\mathcal O}).$ 
An element $X \in M_{ml}({\mathcal O}_p)$ with $m \ge l$ is said to be primitive if there is an element $Y$ of $M_{m,m-l}({\mathcal O}_p)$ such that
$(X \ Y) \in GL_m({\mathcal O}_p).$ If $K_p$ is a field, this is equivalent to saying that ${\rm rank}_{{\mathcal O}_p/\varpi {\mathcal O}_p} X=l.$ 
If $K_p={\bf Q}_p \oplus {\bf Q}_p,$ and $X=(X_1,X_2) \in M_{ml}({\bf Z}_p)  \oplus M_{ml}({\bf Z}_p),$ this is equivalent to saying that 
${\rm rank}_{{\bf Z}_p/p{\bf Z}_p} X_1={\rm rank}_{{\bf Z}_p/p{\bf Z}_p} X_2=l.$
Now let $m$ and $l$ be positive integers such that $m \ge l.$ Then for an integer  $a$ and $A \in \widetilde {\rm Her}_m({\mathcal O}_p), \ B \in \widetilde {\rm Her}_l({\mathcal O}_p)$ put 
$${\mathcal A}_a(A,B)=\{X \in
M_{ml}({\mathcal O}_p)/p^aM_{ml}({\mathcal O}_p) \ | \ A[X]-B \in p^a\widetilde {{\rm Her}}_l({\mathcal O}_p) \},$$
and 
$${\mathcal B}_a(A,B)=\{X \in {\mathcal A}_a(A,B) \ | \ 
  X \ {\rm is \ primitive}  \}.$$
Suppose that $A$ and  $B$ are non-degenerate. Then the number $p^{a(-2ml+l^2)}\#{\mathcal A}_a(A,B)$ is independent of $a$ if $a$ is sufficiently large. Hence we define the local density $\alpha_p(A,B)$  representing $B$ by $A$ as
$$\alpha_p(A,B)=\lim_{a \rightarrow
\infty}p^{a(-2ml+l^2)}\#{\mathcal A}_a(A,B).$$
Similarly  we can define the primitive local density $\beta_p(A,B)$ as 
 $$\beta_p(A,B)=\lim_{a \rightarrow
\infty}p^{a(-2ml+l^2)}\#{\mathcal B}_a(A,B)$$
 if $A$ is non-degenerate. We remark that the primitive local density $\beta_p(A,B)$ can be defined even if $B$ is not non-degenerate. 
In particular we write $\alpha_p(A)=\alpha_p(A,A).$ 
We also define $\upsilon_p(A)$ for $A \in {\rm Her}_m({\mathcal O}_p)^{\times}$ as 
$$\upsilon_p(A)=\lim_{a \rightarrow \infty}p^{-am^2} \#(\Upsilon_{a}(A)),$$
where 
$${\Upsilon}_a(A)=\{X \in
M_{m}({\mathcal O}_p)/p^aM_{m}({\mathcal O}_p) \ | \ A[X]-A \in p^a {\rm Her}_m({\mathcal O}_p) \}.$$
The relation between $\alpha_p(A)$ and $\upsilon_p(A)$ is as follows:

\begin{lem}
Let $T \in \widetilde {\rm Her}_m({\mathcal O}_p)^{\times}.$ Suppose that $K_p$ is ramified over ${\bf Q}_p.$ Then we have
$$\alpha_p(T)=p^{-m(m+1)f_p/2+m^2\delta_{2,p}}\upsilon_p(T).$$
Otherwise, $\alpha_p(T)=\upsilon_p(T).$ 
\end{lem}
\begin{proof}
The proof is similar to that in [Kitaoka \cite{Ki2}, Lemma 5.6.5], and we here give an outline of the proof.
The last assertion is trivial. Suppose that $K_p$ is ramified over ${\bf Q}_p.$ Let $\{T_i \}_{i=1}^l$ be a complete set of representatives of ${\rm Her}_m({\mathcal O}_p)/p^{r+e_p}{\rm Her}_m({\mathcal O}_p)$ such that $T_i \equiv T \ {\rm mod} \
 p^{r}\widetilde {\rm Her}_m({\mathcal O}_p).$ Then it is easily seen that 
 $$l=[p^r\widetilde {\rm Her}_m({\mathcal O}_p):p^{r+e_p}{\rm Her}_m({\mathcal O}_p)]=p^{m(m-1)f_p/2}.$$
 Define a mapping 
 $$\phi:\bigsqcup_{i=1}^{l} \Upsilon_{r+e_p}(T_i) \longrightarrow {\mathcal A}_r(T,T)$$ 
by $\phi(X)=X \ {\rm mod} \ p^r.$ For $X \in {\mathcal A}_r(T,T)$ and $Y \in M_m({\mathcal O}_p)$ we have
 $$T[X+p^rY] \equiv T[X] \ {\rm mod} \ p^r \widetilde {\rm Her}_m({\mathcal O}_p).$$
 Namely, $X+p^rY$ belongs to $\Upsilon_{r+e_p}(T_i)$ for some $i$ and therefore $\phi$ is surjective. Moreover for $X \in {\mathcal A}_r(T,T)$ we have $\#(\phi^{-1}(X))=p^{2m^2e_p}.$ For a sufficiently large integer $r$ we have
 $\#\Upsilon_{r+e_p}(T_i)=\#\Upsilon_{r+e_p}(T)$ for any $i.$ Hence
 $$p^{m(m-1)f_p/2}\#\Upsilon_{r+e_p}(T)=\sum_{i=1}^l \#\Upsilon_{r+e_p}(T_i)$$
 $$=p^{2m^2e_p}\#{\mathcal A}_{r}(T,T)=p^{m^2e_p}\#{\mathcal A}_{r+e_p}(T,T).$$
 Recall that $e_p=f_p-\delta_{2,p}.$ Hence
 $$\#\Upsilon_{r+e_p}(T)=p^{m(m+1)f_p/2-m^2\delta_{2p}}\#{\mathcal A}_{r+e_p}(T,T).$$
 This proves the assertion.
 \end{proof}

\bigskip

For $T \in {\rm Her}_m(K)^+,$ let ${\mathcal G}(T)$ denote the set of $SL_{m}({\mathcal O})$-equivalence classes of  positive definite Hermitian matrices $T'$ such that $T'$ is $SL_m({\mathcal O}_p)$-equivalent to $T$ for any prime number $p.$ Moreover put 
$$M^*(T)= \sum_{T' \in {\mathcal G}(T)} {1 \over e^*(T')}$$
for a positive definite Hermitian matrix $T$ of degree $m$ with entries in ${\mathcal O}.$ 

Let ${\mathcal U}_1$ be the unitary group defined in Section 1. Namely let 
$${\mathcal U}_1=\{ u \in R_{K/{\bf Q}}(GL_1) \ | \ \overline {u} u =1 \}. $$
For an element $T \in {\rm Her}_m({\mathcal O}_p),$ let
$$\widetilde {U_{p,T}}=\{ \det X \ | \ X \in {\mathcal U}_T(K_p) \cap GL_m({\mathcal O}_p) \},$$ 
and put  $U_{1,p}={\mathcal U}_1(K_p)  \cap {\mathcal O}_p^*.$ Then $\widetilde {U_{p,T}}$ is a subgroup of $U_{1,p}$ of finite index. We then put $l_{p,T}=[U_{1,p}:\widetilde {U_{p,T}}].$ 
We also put 
$$u_p=\left\{\begin{array}{ll}
( 1+p^{-1} )^{-1} & \ {\rm if} \ K_p/{\bf Q}_p \ {\rm is \ unramified } \\
 (1-p^{-1)^{-1}} & \ {\rm if} \ K_p={\bf Q}_p \oplus {\bf Q}_p \\
 2^{-1}         & \  \ {\rm if} \ K_p/{\bf Q}_p \ {\rm is \ ramified. }
 \end{array} \right. $$
To state the Mass formula for ${\mathcal SU}_T,$ put $\Gamma_{{\bf C}}(s)=2(2\pi)^{-s} \Gamma(s).$ 
\bigskip
\begin{prop}
Let $T \in {\rm Her}_m({\mathcal O})^+.$ Then
$$M^*(T)={ (\det T)^{m} \prod_{i=2}^{m} D^{i/2}\Gamma_{\bf C}(i)  \over 2^{m-1}\prod_p l_{p,T} u_p \upsilon_p(T)}.$$
 
\end{prop}

\begin{proof}
The assertion is more or less well known (cf. \cite{Re}.)  But for the sake of completeness we 
here give an outline of the proof. Let $\mathcal {SU}_T({\bf A})$ be the adelization of $\mathcal {SU}_T$ and let $\{x_i \}_{i=1}^H$ be a subset of  $\mathcal {SU}_T({\bf A})$  such that
$$\mathcal {SU}_T({\bf A})= \bigsqcup_{i=1}^H  {\mathcal Q} x_i \mathcal {SU}_T({\bf Q}),$$
where ${\mathcal Q}=\mathcal {SU}_T({\bf R}) \prod_{p<\infty} (\mathcal {SU}_T(K_p) \cap SL_m({\mathcal O}_p)).$ We note that the strong approximation theorem holds for $SL_m.$ Hence, by using the standard method we can prove that
$$M^*(T)=\sum_{i=1}^H {1 \over \#(x_i^{-1} \mathcal {Q} x_i \cap \mathcal {SU}_T({\bf Q}))}.$$ 
We recall that the Tamagawa number of $\mathcal {SU}_T$ is 1 (cf. Weil \cite{We}). Hence, by [\cite{Re}, (1.1) and  (4.5)], we have
$$M^*(T)={ (\det T)^{m} \prod_{i=2}^{m} D^{i/2}\Gamma_{\bf C}(i)  \over 2^{m-1}\prod_p l_{p,T}} {\upsilon_p(1) \over \upsilon_p(T)}.$$
We can easily show that $\upsilon_p(1)=u_p^{-1}.$
 This completes the assertion.
\end{proof}

\bigskip

\begin{xcor}
Let $T \in \widetilde {\rm Her}_m({\mathcal O})^+.$ Then
$$M^*(T)={ 2^{c_Dm^2} (\det T)^{m} \prod_{i=2}^{m} \Gamma_{\bf C}(i)  \over 2^{m-1} D^{m(m+1)/4+1/2}\prod_p u_p \l_{p,T}\alpha_p(T)},$$
where $c_D=1$ or $0$ according as $2$ divides $D$ or not.
\end{xcor}

For a subset ${\mathcal T}$ of ${\mathcal O}_p$ put
$${\rm Her}_m({\mathcal T})={\rm Her}_m({\mathcal O}_p) \cap M_m({\mathcal T}),$$ and for a subset ${\mathcal S}$ of ${\mathcal O}_p$ put
$${\rm Her}_m({\mathcal S},{\mathcal T})= \{ A \in {\rm Her}_m({\mathcal T}) \ | \  \det A \in {\mathcal S} \},$$ 
and $\widetilde {\rm Her}_m({\mathcal S},{\mathcal T})={\rm Her}_m({\mathcal S},{\mathcal T}) \cap \widetilde {\rm Her}_m({\mathcal O}_p).$
In particular if ${\mathcal S}$ consists of a single element $d$ we write ${\rm Her}_m({\mathcal S},{\mathcal T})$ as ${\rm Her}_m(d,{\mathcal T}),$ and so on. For  $d \in {\bf Z}_{>0}$ we also define the set $\widetilde {\rm Her}_m(d,{\mathcal O})^+$ in a similar way. 
For each $T \in {\widetilde {\rm Her}}_m({\mathcal O}_p)^{\times}$ put 
$$F_p^{(0)}(T,X)=F_p(p^{-{e}_p}T,X)$$ and 
$$\widetilde F_p^{(0)}(T,X)=\widetilde F_p(p^{-{e}_p}T,X).$$
We remark that 
$$\widetilde F_p^{(0)}(T,X)=X^{-{\rm ord}_p(\det T)}X^{e_pm-f_p[m/2]} F_p^{(0)}(T,p^{-m}X^2).$$
For $d \in {\bf Z}_p^{\times}$ put 
$$\lambda_{m,p}(d,X)=  \sum_{A \in \widetilde {\rm Her}_m(d,{\mathcal O}_p)/SL_{m}({\mathcal O}_p)} {\widetilde F_p^{(0)}(A,X) \over u_pl_{p,A}\alpha_p(A)}.$$
An explicit formula for  $\lambda_{m,p}(p^id_0,X)$ will be given in the next section for $d_0 \in {\bf Z}_p^*$ and
$i \ge 0.$ 

Now let $\widetilde {{\mathsf {Her}}}_m=\prod_p (\widetilde{{\rm Her}}_m({\mathcal O}_p)/SL_m({\mathcal O}_p)).$ Then  the diagonal embedding induces a mapping
$$\phi: \widetilde {{\rm Her}}_m(O)^+/\prod_p SL_m({\mathcal O}_p) \longrightarrow  \widetilde {{\mathsf {Her}}}_m.$$

\bigskip
\begin{prop}
In addition to the above notation and the assumption, for a positive integer $d$ let
$$\widetilde {{\mathsf {Her}}}_m(d)=\prod_p (\widetilde {{\rm Her}}_m(d,{\mathcal O}_p)/SL_m({\mathcal O}_p)).$$ Then  the mapping $\phi$ induces a bijection from $\widetilde {{\rm Her}}_m(d,O)^+/ \prod_p SL_m({\mathcal O}_p)$ to $\widetilde {{\mathsf {Her}}}_m(d),$ which will be denoted also by $\phi.$
\end{prop}

\begin{proof} The proof is similar to that of [\cite{I-S}, Proposition 2.1], but it is a little bit more complex because the class number of $K$ is not necessarily one.  It is easily seen that $\phi$ is injective.
Let $(x_p) \in \widetilde {\mathsf {Her}}_m(d).$  Then by Theorem 6.9 of \cite{Sch},  there exists an element  $y$ in ${\rm Her}_m(K)^+$ such that $\det y  \in dN_{K/{\bf Q}}(K^{\times}).$  Then we have $\det y \in  \det  x_pN_{K_p/{\bf Q}_p}(K_p^{\times})$ for any $p.$ Thus by [\cite{J}, Theorem 3.1]  we have 
  $x_p=g_p^*yg_p$ with some $g_p \in GL_m(K_p)$ for any prime number $p.$  For $p$ not dividing $Dd$ we may suppose $g_p \in GL_m(O_p).$ Hence $(g_p)$ defines an element of $R_{K/{\bf Q}}(GL_m)({\bf A}_f).$ Since we have $d^{-1}\det y \in {\bf Q}^{\times} \cap \prod_p N_{K_p/{\bf Q}_p}(K_p),$ we see that $d^{-1}\det y =N_{K/{\bf Q}}(u)$ with some $u \in K^{\times}.$ 
Thus, by replacing $y$ with $\smallmattwo(1_{m-1};O;O;\overline{u}^{-1})y \smallmattwo(1_{m-1};O;O;u^{-1}),$ we may suppose that $\det y=d.$ Then we have
$N_{K_p/{\bf Q}_p}(\det g_p)=1.$  It is easily seen that there exists an element $\delta_p \in GL_m(K_p)$ such that 
$\det \delta_p=\det g_p^{-1}$ and $\delta_p^* x_p \delta_p=x_p.$ Thus we have $g_p \delta_p \in SL_m(K_p)$ and
$$x_p=(g_p \delta_p)^* y g_p \delta_p.$$ 
By the strong approximation theorem for $SL_m$ there exists an element $\gamma \in SL_m(K), \gamma_{\infty} \in SL_m({\bf C}),$ and $(\gamma_p) \in \prod_p SL_m(O_p)$ such that
$$(g_p \delta_p)= \gamma\gamma_{\infty}(\gamma_p).$$
Put $x=\gamma^* y \gamma.$ Then $x$ belongs to  $\widetilde {\rm Her}_m(d,{\mathcal O})^+,$ and $\phi(x)=(x_p).$ This proves the surjectivity of $\phi.$
\end{proof}

\bigskip

\begin{thm}
 
  Let $f$ be a primitive form in ${\textfrak S}_{2k+1}(\varGamma_0(D),\chi)$  or in ${\textfrak S}_{2k}(SL_2({\bf Z}))$ according as $m=2n$ or $2n+1.$ For such an $f$ and a positive integer $d_0$ put
$$b_m(f;d_0)=\prod_p \lambda_{m,p}(d_0,\alpha_p^{-1}),$$
where $\alpha_p$ is the Satake $p$-parameter of $f.$ 
Moreover put
$$\mu_{m,k,D}=D^{m(s-k+l_0/2)+(k-l_0/2)[m/2]-m(m+1)/4-1/2}$$
$$ \times 2^{-c_Dm(s-k-2n-l_0/2)-m+1} \prod_{i=2}^{m} \Gamma_{\bf C}(i),$$
where $l_0=0$ or $1$ according as $m$ is even or odd. Then for ${\rm Re}(s) >>0,$ we have 
$$L(s,I_m(f))=\mu_{m,k,D} \sum_{d_0=1}^{\infty} b_m(f;d_0) d_0^{-s+k+2n+l_0/2}.$$

\end{thm}

\begin{proof}
We note that $L(s,I_m(f))$ can be rewritten as 
$$L(s,I_m(f))= \widetilde {D}^{ms} \sum_{T \in \widetilde {\rm Her}_m({\mathcal O})^+/SL_m({\mathcal O})} {a_{I_m(f)}({\widetilde D}^{-1}T) \over e^*(T) (\det T)^s }.$$
For $T \in \widetilde {\rm Her}_m({\mathcal O})^+$ the Fourier coefficient $a_{I_m(f)}({\widetilde D}^{-1}T)$ of $I_m(f)$ is uniquely determined by the genus to which $T$ belongs, and can be expressed as
$$a_{I_m(f)}({\widetilde D}^{-1}T)=(D^{[m/2]}\widetilde{D}^{-m}\det T)^{k-l_0/2}\prod_p \widetilde F_p^{(0)}(T,\alpha_p^{-1}).$$
Thus the assertion follows from  Corollary to Proposition 3.2 and Proposition 3.3 similarly to \cite{I-S}.
\end{proof}

\bigskip

\section{Formal power series associated with local Siegel series}
  For $d_0 \in {\bf Z}_p^{\times}$ put 
$$\hat P_{m,p}(d_0,X,t)= \sum_{i=0}^{\infty}\lambda_{m,p}^*(p^id_0,X)t^i,$$
where for $d \in {\bf Z}_p^{\times}$ we define $\lambda_{m,p}^*(d,X)$ as
$$\lambda_{m,p}^*(d,X)=
\sum_{A \in {\widetilde{\rm Her}}_m(dN_{K_p/{\bf Q}_p}({\mathcal O}_p^*),{\mathcal O}_p)/GL_{m}({\mathcal O}_p)}  {\widetilde F_p^{(0)}(A,X) \over \alpha_p(A)}.$$
We note that 
$$\sum_{A \in {\widetilde{\rm Her}}_m(dN_{K_p/{\bf Q}_p}({\mathcal O}_p^*),{\mathcal O}_p)/GL_{m}({\mathcal O}_p)}  {\widetilde F_p^{(0)}(A,X^{-1}) \over \alpha_p(A)}$$
 is $\chi_{K_p}((-1)^{m/2}d)\lambda_{m,p}^*(d,X)$ or $\lambda_{m,p}^*(d,X)$
according as $m$ is even and $K_p$ is a field,  or not.
In  Proposition 4.3.7 we will show that  we have  
$$\lambda_{m,p}^*(d,X)=u_p\lambda_{m,p}(d,X)$$
for $d \in {\bf Z}_p^{\times}$ and therefore   
$$\hat P_{m,p}(d_0,X,t)=u_p \sum_{i=0}^{\infty}\lambda_{m,p}(p^id_0,X)t^i.$$
We also define  $P_{m,p}(d_0,X,t)$ as 
$$P_{m,p}(d_0,X,t)= \sum_{i=0}^{\infty}\lambda_{m,p}^*(\pi_p^id_0,X)t^i.$$
We note that $P_{m,p}(d_0,X,t)=\hat P_{m,p}(d_0,X,t)$ if $K_p$ is unramified over ${\bf Q}_p$ or $K_p={\bf Q}_p \oplus {\bf Q}_p,$ but it is not necessarily the case if $K_p$ is ramified over ${\bf Q}_p.$
In this section, we give  explicit formulas of $P_{m,p}(d_0,X,t)$ for all prime numbers $p$ (cf. Theorems 4.3.1 and 4.3.2),
and therefore explicit formulas for $\hat P_{m,p}(d_0,X,t)$ (cf. Theorem 4.3.6.)

From now on we fix a prime number $p.$ Throughout this section we simply write ${\rm ord}_p$ as ${\rm ord}$ and so on if the prime number $p$ is clear from the context. We also write $\nu_{K_p}$ as $\nu.$ We also simply write $\widetilde {\rm Her}_{m,p}$ instead of $\widetilde{\rm Her}_m({\mathcal O}_p),$ and so on.

\subsection{Preliminaries}

\noindent
{ }

\bigskip

 Let $m$ be a positive integer. For a non-negative integer $i \le m$ let 
$${\mathcal D}_{m,i}=GL_m({\mathcal O}_p) \mattwo(1_{m-i};0;0;\varpi 1_i) GL_m({\mathcal O}_p),$$
 and for $W \in  {\mathcal D}_{m,i},$ put $\varPi_p(W)=(-1)^i p^{i(i-1)a/2},$
 where $a=2$ or $1$ according as $K_p$ is unramified over ${\bf Q}_p$ or not.
Let $K_p={\bf Q}_p \oplus {\bf Q}_p.$ Then for a pair $i=(i_1,i_2)$ of non-negative integers such that $i_1,i_2 \le m,$ let 
$${\mathcal D}_{m,i}=GL_m({\mathcal O}_p) \left(\mattwo(1_{m-i_1};0;0;p 1_{i_1}),\mattwo(1_{m-i_2};0;0;p 1_{i_2}) \right) GL_m({\mathcal O}_p),$$
 and for $W \in {\mathcal D}_{m,i}$ put $\varPi_p(W)=(-1)^{i_1+i_2} p^{i_1(i_1-1)/2+i_2(i_2-1)/2}.$ In either case $K_p$ is a quadratic extension of ${\bf Q}_p,$ or  $K_p={\bf Q}_p \oplus {\bf Q}_p,$ we put $\varPi_p(W)=0$ for $W \in M_n({\mathcal O}_p^{\times}) \setminus \bigcup_{i=0}^m {\mathcal D}_{m,i}.$ 

First we remark the following lemma, which can easily be proved by the usual Newton approximation method in ${\mathcal O}_p$:

\bigskip

\begin{lems}
Let $A,B \in \widetilde {\rm Her}_m({\mathcal O}_p)^{\times}.$ Let $e$ be an integer such that $p^{e}A^{-1} \in \widetilde {\rm Her}_m({\mathcal O}_p).$ Suppose that 
$A \equiv B \ {\rm mod} \ p^{e+1}\widetilde {\rm Her}_m({\mathcal O}_p).$
Then there exists a matrix  $U \in GL_m({\mathcal O}_p)$ such that 
$B=A[U].$
\end{lems}

\bigskip

\bigskip
\begin{lems}
Let $S \in \widetilde {\rm Her}_m({\mathcal O}_p)^{\times}$ and $T \in \widetilde {\rm Her}_n({\mathcal O}_p)^{\times}$ with $m \ge n.$ Then 
$$\alpha_p(S,T)=\sum_{W \in GL_n({\mathcal O}_p) \backslash M_n({\mathcal O}_p)^{\times}} p^{(n-m)\nu(\det W)}\beta_p(S,T[W^{-1}]).$$
\end{lems}

\begin{proof} 
The assertion can be proved by using the same argument as in the proof of [\cite{Ki2}, Theorem 5.6.1].
We here give an outline of the proof. For each $W \in M_n({\mathcal O}_p),$ put
$${\mathcal B}_e(S,T;W)=\{X \in {\mathcal A}_e(S,T) \ | \ XW^{-1} \ {\rm is \ primitive} \}.$$
Then we have
$${\mathcal A}_e(S,T)=\bigsqcup_{W \in GL_n({\mathcal O}_p) \backslash M_n({\mathcal O}_p)^{\times}} {\mathcal B}_e(S,T;W).$$
Take a sufficiently large integer $e,$ and for an element  $W$ of $M_n({\mathcal O}_p),$ let $\{R_i\}_{i=1}^r$ be a complete set of representatives of 
$p^e \widetilde {\rm Her}_m({\mathcal O}_p)[W^{-1}]/p^e\widetilde {\rm Her}_m({\mathcal O}_p).$ Then 
we have  $r=p^{\nu(\det W)n}.$ Put
$$\widetilde {\mathcal B}_e(S,T;W)=\{X \in M_{mn}({\mathcal O}_p)/p^eM_{mn}({\mathcal O}_p)W \ | \ S[X] \equiv T \ {\rm mod} \ p^e \widetilde {\rm Her}_m({\mathcal O}_p) $$
$$ \ {\rm and} \ XW^{-1} \ {\rm is \ primitive} \}.$$ 
Then 
$$\#(\widetilde {\mathcal B}_e(S,T;W))=p^{\nu(\det W)m}\#({\mathcal B}_e(S,T;W)).$$
It is easily seen that 
$$S[XW^{-1}] \equiv T[W^{-1}] +R_i \ {\rm mod} \ p^e \widetilde {\rm Her}_m({\mathcal O}_p)$$
for some $i.$ Hence the mapping $X \mapsto XW^{-1}$ induces a bijection from 
$\widetilde {\mathcal B}_e(S,T;W)$ to $\displaystyle \bigsqcup_{i=1}^r {\mathcal B}_e(S,T[W^{-1}]+R_i).$ Recall that $\nu(W) \le {\rm ord}(\det T).$ Hence 
$$R_i \equiv O \ {\rm mod} \ p^{[e/2]}\widetilde {\rm Her}_m({\mathcal O}_p),$$
and therefore by Lemma 4.1.1, 
$$T[W^{-1}]+R_i =T[W^{-1}][G]$$
 for some $G \in GL_n({\mathcal O}_p).$ Hence
 $$\#(\widetilde {\mathcal B}_e(S,T;W))=p^{\nu(\det W)n}\#({\mathcal B}_e(S,T[W^{-1}])).$$
 Hence
 $$\alpha_p(S,T)=p^{-2mne+n^2e}\#({\mathcal A}_e(S,T))$$
 $$=p^{-2mne+n^2e}\sum_{W \in GL_n({\mathcal O}_p) \backslash M_n({\mathcal O}_p)^{\times}}   p^{\nu(\det W)(-m+n)}\#({\mathcal B}_e(S,T[W^{-1}])).$$
 This proves the assertion.

\end{proof}

\bigskip
Now  by using the same argument as in the proof of [\cite{Ki0}, Theorem 1],  we obtain

\begin{xcor}
Under the same notation as above, we have
$$\beta_p(S,T)=\sum_{W \in GL_n({\mathcal O}_p) \backslash M_n({\mathcal O}_p)^{\times}} p^{(n-m)\nu(\det W)}\varPi_p(W)\alpha_p(S,T[W^{-1}]).$$
\end{xcor}

\bigskip
For two elements $A,A' \in {\rm Her}_m({\mathcal O}_p)$ we simply write $A \sim_{GL_m({\mathcal O}_p)} A'$ as $A \sim A'$ if there is no fear of confusion. 
For a variables $U$ and $q$ put
$$(U,q)_m=\prod_{i-1}^{m}(1-q^{i-1}U), \qquad  \phi_m(q)=(q,q)_m.$$
We note that $\phi_m(q)=\prod_{i=1}^m (1-q^i).$
Moreover for a prime number $p$ put
 $$\phi_{m,p}(q)=\left\{\begin{array}{ll}
 \phi_m(q^2) \ & \ {\rm if} \ K_p/{\bf Q}_p \ {\rm is \ unramified} \\
 \phi_m(q)^2 \ & \ {\rm if} \ K_p={\bf Q}_p \oplus {\bf Q}_p \\
 \phi_m(q)   \ & \ {\rm if} \  K_p/{\bf Q}_p \ {\rm is \ ramified}
 \end{array}
 \right.$$

\bigskip

\begin{lems}
 {\rm (1)} Let $\Omega(S,T)=\{w \in M_m({\mathcal O}_p) \ | \ S[w] \sim T \}.$ 
 Then we have
$${\alpha_p(S,T) \over \alpha_p(T)}=\#(\Omega(S,T)/GL_m({\mathcal O}_p))p^{-m({\rm ord}(\det T)- {\rm ord}(\det S))}.$$

{\rm (2)} Let $\widetilde \Omega(S,T)=\{w \in M_m({\bf Z}) \ | \ S \sim T[w^{-1}] \}.$ 
Then we have  
$${\alpha_p(S,T) \over \alpha_p(S)}=\#(GL_m({\mathcal O}_p) \backslash \widetilde \Omega(S,T)).$$
\end{lems}

\bigskip

\begin{proof} (1) The proof is similar to that of Lemma 2.2 of \cite{B-S}. First we prove
$$\int_{\Omega(S,T)} |dx| =\phi_{m,p}(p^{-1}){\alpha_p(S,T) \over \alpha_p(T)},$$
where $|dx|$ is the Haar measure on $M_{m}(K_p)$ normalized so that 
$$\int_{M_{m}({\mathcal O}_p)}|dx|=1.$$
To prove this, for a positive integer $e$ let
$T_1,...,T_l$ be a complete set of representatives of $\{T[\gamma] \ {\rm mod} \ p^e \ | \ \gamma \in GL_m({\mathcal O}_p)\}.$ Then it is easy to see that
$$\int_{\Omega(S,T)} |dx|=p^{-2m^2e}\sum_{i=1}^l \#({\mathcal A}_e(S,T_i))$$
and, by Lemma 4.1.1,  $T_i$ is $GL_m({\mathcal O}_p)$-equivalent to $T$ if $e$ is sufficiently large. Hence we have 
$$\#({\mathcal A}_e(S,T_i))=\#({\mathcal A}_e(S,T))$$
for any $i.$ Moreover we have
$$l=\#(GL_m({\mathcal O}_p/p^e{\mathcal O}_p))/\#({\mathcal A}_e(T,T))=p^{m^2e}\phi_{m,p}(p^{-1})/\alpha_p(T).$$
Hence 
$$\int_{\Omega(S,T)} |dx| =lp^{-2m^2e}\#({\mathcal A}_e(S,T))=\phi_{m,p}(p^{-1}){\alpha_p(S,T) \over \alpha_p(T)},$$
which proves the above equality. Now we have
$$\int_{\Omega(S,T)} |dx| =\sum_{W \in \Omega(S,T)/GL_m({\mathcal O}_p)} |\det W|_{K_p}^m=\sum_{W \in \Omega(S,T)/GL_m({\mathcal O}_p)} |\det W \overline{\det W}|_p^m.$$
Remark that for any $W \in \Omega(S,T)/GL_m({\mathcal O}_p)$ we have
$|\det W \overline{\det W}|_p=p^{-m({\rm ord}(\det T)- {\rm ord}(\det S))}.$
Thus the assertion has been proved.

(2) By Lemma 4.1.2 we have
$$\alpha_p(S,T)=\sum_{W \in GL_m({\mathcal O}_p) \backslash  M_m({\mathcal O}_p)^{\times}} \beta_p(S,T[W^{-1}]).$$
Then we have 
$\beta_p(S,T[W^{-1}])=\alpha_p(S)$ or $0$ according as $S \sim T[W^{-1}]$ or not.
Thus the assertion (2) holds.
\end{proof}

\bigskip

For a subset ${\mathcal T}$ of ${\mathcal O}_p,$ we put 
$${\rm Her}_m({\mathcal T})_k=\{ A=(a_{ij})  \in {\rm Her}_m({\mathcal T}) \ | \  a_{ii} \in \pi^k{\bf Z}_p \}.$$
 
 From now on put
 $${\rm Her}_{m,*}({\mathcal O}_p)=\left\{\begin{array}{ll}
  {\rm Her}_{m}({\mathcal O}_p)_1 \ & {\rm if} \ p=2 \ {\rm and} \ f_p=3, \\
 {\rm Her}_{m}(\varpi {\mathcal O}_p)_1 \ & {\rm if} \ p=2 \ {\rm and} \ f_p=2 \\
 {\rm Her}_{m}({\mathcal O}_p) \ & \ {\rm  otherwise,} 
 \end{array}
 \right.$$
 where $\varpi$ is a prime element of $K_p.$ 
Moreover put  $i_p=0,$ or $1$ according as $p=2$ and $f_2=2,$ or not.
 Suppose that $K_p/{\bf Q}_p$ is unramified  or $K_p={\bf Q}_p \oplus {\bf Q}_p.$  Then an element $B$ of $\widetilde{\rm Her}_{m}({\mathcal O}_p)$ can be expressed as 
 $B \sim_{GL_m({\mathcal O}_p)} 1_r \bot pB_2$ with some integer $r$ and $B_2 \in {\rm Her}_{m-r,*}({\mathcal O}_p).$ 
Suppose that $K_p/{\bf Q}_p$ is ramified. For an even positive integer $r$ define $\Theta_r$ by 
$$\Theta_{r}= \overbrace{\mattwo(0;\varpi^{i_p};\overline {\varpi}^{i_p};0) \bot ...\bot \mattwo(0;\varpi^{i_p};\overline{\varpi}^{i_p};0)}^{r/2},$$
where $\overline{\varpi}$ is the conjugate of $\varpi$ over ${\bf Q}_p.$
 Then an element $B$ of $\widetilde{\rm Her}_{m}({\mathcal O}_p)$ is expressed as $B \sim_{GL_m({\mathcal O}_p)} \Theta_r \bot \pi^{i_p}B_2$ with some even integer $r$ and $B_2 \in {\rm Her}_{m-r,*}({\mathcal O}_p).$
For these results, see Jacobowitz \cite{J}.





A non-degenerate square matrix $W=(d_{ij})_{m \times m}$ with entries in ${\mathcal O}_p$ is called reduced if $W$ satisfies the following conditions:

$d_{ii}=p^{e_{i}}$ with $e_i$ a non-negative integer, $d_{ij}$ is a non-negative integer $\le p^{e_j}-1$ for $i <j$ and $d_{ij}=0$ for $i >j.$ 
It is well known that  we can take the set of all reduced matrices as a  complete set of representatives of $GL_m({\mathcal O}_p) \backslash M_m({\mathcal O}_p)^{\times}.$ Let $m$ be an integer.  For $B \in {\widetilde{\rm Her}}_{m}({\mathcal O}_p)$ put
$$\widetilde \Omega(B)=\{W \in GL_m(K_p) \cap M_m({\mathcal O}_p) \ | \ B[W^{-1}] \in \widetilde{\rm Her}_{m}({\mathcal O}_p) \}.$$
 Let $r \le m,$ and $\psi_{r,m}$ be the mapping from $GL_{r}(K_p)$ into $GL_{m}(K_p)$ defined by $\psi_{r,m}(W)=1_{m-r} \bot W.$ 

\bigskip
\begin{lems}

{\rm (1)} Assume that  $K_p$ is unramified over  ${\bf Q}_p$ or $K_p={\bf Q}_p \oplus {\bf Q}_p.$ Let $B_1 \in {\rm Her}_{m-n_0}({\mathcal O}_p).$ Then $\psi_{m-n_0,m}$ induces a bijection from $GL_{m-n_0}({\mathcal O}_p) \backslash \widetilde \Omega(B_1)$ to $GL_{m}({\mathcal O}_p) \backslash \widetilde \Omega(1_{n_0} \bot B_1),$
which will be also denoted by $\psi_{m-n_0,m}.$ 

\noindent
{\rm (2)} Assume that  $K_p$ is ramified over ${\bf Q}_p$ and that $n_0$ is even. Let $B_1 \in \widetilde {\rm Her}_{m-n_0}({\mathcal O}_p).$ Then $\psi_{m-n_0,m}$ induces a bijection from 
$GL_{m-n_0}({\mathcal O}_p) \backslash \widetilde \Omega(B_1)$ to $GL_{m}({\mathcal O}_p) \backslash \widetilde \Omega(\Theta_{n_0} \bot B_1),$
which will be also denoted by $\psi_{m-n_0,m}.$ Here $i_p$ is the integer defined above.
\end{lems}
\begin{proof} (1) Clearly $\psi_{m-n_0,m}$ is injective. To prove the surjectivity, take a representative $W$ of an element of $GL_{m}({\mathcal O}_p) \backslash \widetilde \Omega(1_{n_0} \bot B_1).$ Without loss of generality we may assume that $W$ is a reduced matrix.
Since we have $(1_{n_0} \bot B_1)[W^{-1}]  \in  \widetilde{{\rm Her}_m}({\mathcal O}_p),$ we have 
$W=\mattwo(1_{n_0};0;0;W_1)$ with $W_1 \in \widetilde \Omega(B_1).$ This proves the assertion.

(2) The assertion can be proved in the same manner as (1).
\end{proof}

  \begin{lems}
  Let $B \in \widetilde {\rm Her}_{m}({\mathcal O}_p)^{\times}.$ Then we have
$$\alpha_p(\pi^r d B)=p^{rm^2}\alpha_p(B)$$
for any non-negative integer $r$ and $d \in {\bf Z}_p^*.$ 
   \end{lems}
 
\begin{proof} The assertion can be proved by using the same argument as in the proof of  (a) of Theorem 5.6.4 of Kitaoka [Ki2].
\end{proof}

Now we prove induction formulas for local densities different from Lemma 4.1.2 (cf. Lemmas 4.1.6,  4.1.7, and 4.1.8.) For technical reason, we formulate and prove them  in  terms of Hermitian modules. Let $M$ be  ${\mathcal O}_p$ free module, and let $b$ be a mapping from $M \times M$ to $K_p$ such that
$$b(\lambda_1 u + \lambda_2 u_2,v)=\lambda_1 b(u_1,v) + \lambda_2b(u_2,v)$$
for $u,v \in M $ and $\lambda_1,\lambda_2 \in {\mathcal O}_p,$ and
 $$b(u,v)=\overline {b(v,u)} \ {\rm for} \ u,v \in M.$$
 We call such an $M$ a Hermitian module with a Hermitian inner product $b.$ We set $q(u)=b(u,u)$ for $u \in M.$ Take an ${\mathcal O}_p$-basis $\{u_i\}_{i=1}^m$ of $M,$ and put $T_M=(b(u_i,u_j))_{1 \le i,j \le m}.$ Then $T_M$ is a Hermitian matrix, and its determinant is uniquely determined, up to $N_{K_p/{\bf Q}_p}({\mathcal O}_p^*),$  by $M.$ We say $M$ is non-degenerate if $\det T_M \not= 0.$ Conversely for a Hermitian matrix $T$ of degree $m,$ we can define a Hermitian module $M_T$ so that
$$M_T={\mathcal O}_pu_1 + {\mathcal O}_pu_2 \cdots + {\mathcal O}_pu_m$$
with  $(b(u_i,u_j))_{1 \le i,j \le m}=T.$
Let $M_1$ and $M_2$ be submodules of $M$. We then write $M=M_1 \bot M_2$ if 
$M=M_1 + M_2,$ and $b(u,v)=0$ for any $u \in M_1, v \in M_2.$
Let $M$ and $N$ be Hermitian modules. Then a homomorphism $\sigma:N \longrightarrow M$ is said to be an isometry if $\sigma$ is injective and $b(\sigma(u),\sigma(v))=b(u,v)$ for any $u,v \in N.$ In particular $M$ is said to be isometric to $N$ if $\sigma$ is an isomorphism.
We denote by $U_M'$ the group of isometries of $M$ to $M$ itself. From now on we assume that $T_M \in \widetilde {\rm Her}_m({\mathcal O}_p)$ for a Hermitian module $M$ of rank $m.$ For Hermitian modules $M$ and $N$ over ${\mathcal O}_p$ of rank $m$ and $n$ respectively,
 put
 $${\mathcal A}_a'(N,M)=\{\sigma:N \longrightarrow M/p^aM \ | \ q(\sigma(u)) \equiv q(u) \ {\rm mod} \ p^{e_p+a} \},$$
 and 
 $${\mathcal B}_a'(N,M):=\{\sigma \in {\mathcal A}_a'(N,M) \ | \ \sigma \ {\rm is \ primitive}\}.$$
 Here a homomorphism $\sigma:N \longrightarrow M$ is said to be primitive if $\phi$ induces an injective mapping from $N/\varpi N$ to $M/\varpi M.$ Then we can define the local density
 $\alpha_p'(N,M)$ as 
$$\alpha_p'(N,M)=\lim_{a \rightarrow \infty} p^{-a(2mn-n^2)}\#({\mathcal A}_a'(N,M))$$
if $M$ and $N$ are non-degenerate, and the primitive local density $\beta_p'(N,M)$ as 
$$\beta_p'(N,M)=\lim_{a \rightarrow \infty} p^{-a(2mn-n^2)}\#({\mathcal B}_a'(N,M))$$
if $M$ is non-degenerate as in the matrix case.
 It is easily seen that 
 $$\alpha_p(S,T)=\alpha_p'(M_T,M_S),$$
 and 
 $$\beta_p(S,T)=\beta_p'(M_T,M_S).$$
Let $N_1$ be a submodule of $N.$ For each  $\phi_1 \in {\mathcal B}'_a(N_1,M),$ put
$${\mathcal B}'_a(N,M;\phi_1)=\{\phi \in {\mathcal B}'_a(N,M) \ | \ \phi |_{N_1}=\phi_1 \}.$$
We note that we have 
$${\mathcal B}'_a(N_,M)=\bigsqcup_{\phi_1 \in  {\mathcal B}'_a(N_1,M)} {\mathcal B}'_a(N_,M;\phi_1).$$

\bigskip

Suppose that $K_p$ is unramified over ${\bf Q}_p.$ Then put $\Xi_{m}=1_{m}.$
Suppose that $K_p$ is ramified over ${\bf Q}_p,$ and that $m$ is even. Then put
$\Xi_{m}=\Theta_{m}.$ 

\bigskip

\begin{lems}

Let $m_1,m_2,n_1,$ and $n_2$ be non-negative integers
 such that $m_1 \ge n_1$ and $m_1 + m_2 \ge n_1 +n_2.$ Moreover suppose that $m_1$ and $n_1$ are even if $K_p$ is ramified over ${\bf Q}_p.$
Let $A_2 \in \widetilde {\rm Her}_{m_2}({\mathcal O}_p)$ and  $B_2 \in \widetilde {\rm Her}_{n_2}({\mathcal O}_p).$ 
Then we have
$$\beta_p( \Xi_{m_1} \bot A_2, \Xi_{n_1} \bot B_2)=\beta_p(\Xi_{m_1} \bot A_2,  \Xi_{n_1})\beta_p(\Xi_{m_1-n_1} \bot A_2,B_2),$$
and in particular we have 
$$\beta_p( \Xi_{n_1} \bot A_2, \Xi_{n_1} \bot B_2)=\beta_p(\Xi_{n_1} \bot A_2,  \Xi_{n_1})\beta_p(A_2,B_2),$$
\end{lems}
\begin{proof}
Let $M=M_{\Xi_{m_1} \bot A_2}, N_1=M_{\Xi_{n_1}},N_2=M_{B_2},$  and $N=N_1 \bot N_2.$ Let $a$ be a  sufficiently large positive integer.  
Let $N_1={\mathcal O}_pv_1 \oplus \cdots \oplus {\mathcal O}_pv_{n_1}$ and $ N_2={\mathcal O}_pv_{n_1+1} \oplus \cdots \oplus {\mathcal O}_pv_{n_1+n_2}.$
For each  $\phi_1 \in {\mathcal B}'_a(N_1,M),$ put $u_i=\phi_1(v_i)$ for $i=1,\cdots,n_1.$ 
 Then we can take elements $u_{n_1+1},\cdots,u_{m_1+m_2} \in M$ such that
$$(u_i,u_j)=0 \ (i=1,\cdots,n_1, \ j=n_1+1,\cdots,m_1+m_2),$$ 
and
$$((u_i,u_j))_{n_1+1 \le i,j \le m_1+m_2}= \Xi_{m_1-n_1} \bot A_2.$$
Put $N_1'={\mathcal O}_pu_1 \oplus \cdots \oplus {\mathcal O}_pu_{n_1}.$
Then we have $N_1'=M_{\Xi_{n_1}}.$ 
For $\phi \in {\mathcal B}'_a(N_1,M;\phi_1)$ and $i=1,\cdots,n_2$ we have
$$\phi(v_{n_1+i})=\sum_{j=1}^{m_1+m_2} a_{n_1+i,j} u_j$$ 
with $a_{n_1+i,j} \in {\mathcal O}_p.$ 
Put $\Xi_{n_1}=(b_{ij})_{1 \le i,j \le n_1}.$ Then we have
$$(\phi(v_j),\phi(v_{n_1+i}))=\sum_{\gamma=1}^{n_1}  \overline {a_{n_1+i,\gamma}}  b_{j \gamma}=0$$
for $i=1,\cdots,n_2$ and $j=1,\cdots,n_1.$ 
Hence we have $a_{n_1+i,\gamma}=0$ for $i=1,\cdots,n_2$ and $\gamma=1,\cdots,n_1.$
This implies that $\phi|_{N_2} \in {\mathcal B}'_a(N_2,M_{A_2 \bot \Xi_{m_1-n_1}}).$
Then the mapping 
$${\mathcal B}'_a(N_1,M;\phi_1) \ni \phi \mapsto \phi|_{N_2} \in {\mathcal B}'_a(N_2,M_{A_2 \bot \Xi_{m-n_1}})$$
is bijective. 
Thus we have
$$\#{\mathcal B}'_a(N,M)=\#{\mathcal B}'_a(N_1,M) \#{\mathcal B}'_a(N_2 , M_{ \Xi_{m-n_1} \bot A_2}).$$
This implies that 
$$\beta_p(\Xi_{m_1} \bot A_2,\Xi_{n_1}\bot B_2)=\beta_p(\Xi_{m_1} \bot A_2,\Xi_{n_1})\beta_p( \Xi_{m_1-n_1} \bot A,B_2).$$
\end{proof}

\bigskip

\begin{lems}
In addition to the notation and the assumption in Lemma 4.1.6, suppose that $A_1$ and $A_2$ are non-degenerate. 
Then 
$$\alpha_p(\Xi_{m_1} \bot A_2,  \Xi_{n_1})=\beta_p(\Xi_{m_1} \bot A_2,  \Xi_{n_1}),$$
and we have
$$\alpha_p( \Xi_{m_1} \bot A_2, \Xi_{n_1} \bot B_2)=\alpha_p(\Xi_{m_1} \bot A_2,  \Xi_{n_1})\alpha_p(\Xi_{m_1-n_1} \bot A_2,B_2),$$
and in particular we have 
$$\alpha_p( \Xi_{n_1} \bot A_2, \Xi_{n_1} \bot B_2)=\alpha_p(\Xi_{n_1} \bot A_2,  \Xi_{n_1})\alpha_p(A_2,B_2),$$
\end{lems}
\begin{proof}
The first assertion can easily be proved. By Lemmas 4.1.2 and 4.1.4, we have
$$\alpha_p( \Xi_{m_1} \bot A_2, \Xi_{n_1} \bot B_2)$$
$$=\sum_{W \in GL_{n_1+n_2}({\mathcal O}_p) \backslash \widetilde \Omega(\Xi_{n_1} \bot B_2)} p^{(n_1+n_2-(m_1+m_2))\nu(\det W)} \beta_p( \Xi_{m_1} \bot A_2, (\Xi_{n_1} \bot B_2)[W^{-1}])$$
$$=\sum_{X \in GL_{n_2}({\mathcal O}_p) \backslash \widetilde \Omega(B_2)} p^{(n_2-(m_1-n_1+m_2))\nu(\det X)}\beta_p( \Xi_{m_1} \bot A_2, \Xi_{n_1} \bot B_2 [X^{-1}]).$$
By Lemma 4.1.6 and the first assertion, we have 
$$\beta_p( \Xi_{m_1} \bot A_2, \Xi_{n_1} \bot B_2 [X^{-1}])=\alpha_p(\Xi_{m_1} \bot A_2, \Xi_{n_1})\beta_p(\Xi_{m_1-n_1} \bot A_2,B_2 [X^{-1}]).$$
Hence again by Lemma 4.1.2, we prove the second assertion. 
\end{proof}

\bigskip

\begin{lems}
{\rm (1)} Suppose that $K_p$ is unramified over ${\bf Q}_p.$ Let $A \in {\rm Her}_l({\mathcal O}_p), B_1 \in {\rm Her}_{n_1}({\mathcal O}_p)$
and $B_2 \in {\rm Her}_{n_2}({\mathcal O}_p)$  with $m \ge 2n_1.$ \\
Then we have
$$\beta_p(1_{m} \bot A, B_1 \bot B_2)=\beta_p(1_{m} \bot A,B_1)\beta_p( (-B_1) \bot  1_{m-2n_1} \bot A,  B_2)$$
{\rm (2)} Suppose that $K_p$ is ramified over ${\bf Q}_p.$ Let $A \in \widetilde {\rm Her}_{l}({\mathcal O}_p),  B_1 \in \widetilde {\rm Her}_{n_1}({\mathcal O}_p),$
and $B_2 \in \widetilde {\rm Her}_{n_2}({\mathcal O}_p)$ with $m \ge n_1.$
Then we have
$$\beta_p(\Theta_{2m} \bot A ,B_1 \bot B_2)=\beta_p(\Theta_{2m} \bot A, B_1)\beta_p( (- B_1) \bot \Theta_{2m-2n_1} \bot A,B_2).$$
\end{lems}
\begin{proof}
First suppose that $K_p$ is ramified over ${\bf Q}_p.$ Let $M=M_{\Theta_{2m} \bot A}, N_1=M_{B_1},N_2=M_{B_2},$  and $N=N_1 \bot N_2.$ Let $a$ be a  sufficiently large positive integer.  
Let $N_1={\mathcal O}_pv_1 \oplus \cdots \oplus {\mathcal O}_pv_{n_1}$ and $ N_2={\mathcal O}_pv_{n_1+1} \oplus \cdots \oplus {\mathcal O}_pv_{n_1+n_2}.$
For each  $\phi_1 \in {\mathcal B}'_a(N_1,M),$ put $u_i=\phi_1(v_i)$ for $i=1,\cdots,n_1.$ 
 Then we can take elements $u_{n_1+1},\cdots,u_{2m+l} \in M$ such that
$$(u_i,u_{n_1+j})=\delta_{ij}\varpi^{i_p}, \ (u_{n_1+i},u_{n_1+j})=0 \ (i,j=1,\cdots,n_1),$$ 
$$(u_i,u_j)=0 \ (i=1,\cdots,2n_1, j=2n_1+1,\cdots,2m+l),$$
and
$$((u_i,u_j))_{2n_1+1 \le i,j \le 2m+l}= \Theta_{2m-2n_1} \bot A,$$
where $\delta_{ij}$ is Kronecker's delta. 
Let $B_1=(b_{ij})_{1 \le i,j \le n_1},$
and put
$$u_j'=u_j -\bar \varpi^{-i_p}\sum_{\gamma=1}^{n_1} \bar b_{\gamma j}u_{n_1+\gamma}$$
for $j=1,\cdots,n_1,$ 
and $M'={\mathcal O}_pu_1' \oplus \cdots \oplus {\mathcal O}_pu_{n_1}'.$
Then we have $(u_i',u_j')=-b_{ij}$ and hence we have $M'=M_{(-B_1)}.$ 
For $\phi \in {\mathcal B}'_a(N_1,M;\phi_1)$ and $i=1,\cdots,n_2$ we have
$$\phi(v_{n_1+i})=\sum_{j=1}^{2m+l} a_{n_1+i,j} u_j$$ 
with $a_{n_1+i,j} \in {\mathcal O}_p.$ 
 Then we have
$$(\phi(v_j),\phi(v_{n_1+i}))=\sum_{\gamma=1}^{n_1}  \overline {a_{n_1+i,\gamma}}  b_{j \gamma}+\overline{a_{n_1+i,n_1+j}}\varpi^{i_p}=0$$
for $i=1,\cdots,n_2$ and $j=1,\cdots,n_1.$ Hence we have
$$\phi(v_{n_1+i})=\sum_{j=1}^{n_1} a_{n_1+i,j} u_j' +\sum_{j=2n_1+1}^{2m+l} a_{n_1+i,j} u_j.$$ 
This implies that $\phi|_{N_2} \in {\mathcal B}'_a(N_2, M_{(-B_1)} \bot M_{A \bot \Theta_{2m-2n_1}}).$
Then the mapping 
$${\mathcal B}'_a(N_1,M;\phi_1) \ni \phi \mapsto \phi|_{N_2} \in {\mathcal B}'_a(N_2, M_{(-B_1)} \bot M_{A \bot \Theta_{2m-2n_1}})$$
is bijective. 
Thus we have
$$\#{\mathcal B}'_a(N,M)=\#{\mathcal B}'_a(N_1,M) \#{\mathcal B}'_a(N_2 , M_{(- B_1)} \bot M_{ \Theta_{2m-2n_1} \bot A}).$$
This implies that 
$$\beta_p(\Theta_{2m} \bot A,B_1 \bot B_2)=\beta_p(\Theta_{2m} \bot A,B_1)\beta_p( (-B_1) \bot  \Theta_{2m-2n_1} \bot A,B_2).$$
This proves (2). Next suppose that $K_p$ is unramified over ${\bf Q}_p.$ 
For an even positive integer $r$ define $\Theta_r$ by 
$$\Theta_{r}= \overbrace{\mattwo(0;1;1;0) \bot ...\bot \mattwo(0;1;1;0)}^{r/2}.$$
Then we have $\Theta_r \sim 1_r.$ By using the same argument as above we can prove that
$$\beta_p(\Theta_{m} \bot A ,B_1 \bot B_2)=\beta_p(\Theta_{m} \bot A, B_1)\beta_p( (- B_1) \bot \Theta_{m-2n_1} \bot A,B_2)$$
or 
$$\beta_p(\Theta_{m-1} \bot  1 \bot A ,B_1 \bot B_2)=\beta_p(\Theta_{m-1} \bot 1 \bot A, B_1)\beta_p( (- B_1) \bot \Theta_{m-2n_1} \bot 1 \bot A,B_2)$$
according as $m$ is even or not. Thus we prove the assertion (1).

\end{proof}

\begin{lems}
Let $k$ be a positive integer.

\noindent
{\rm (1)} Suppose that $K_p$ is unramified over ${\bf Q}_p.$ \\
{\rm  (1.1)} Let $b \in {\bf Z}_p.$ Then  we have
$$\beta_p(1_{2k},pb)=   (1- p^{-2k})(1+p^{-2k+1}).$$
{\rm  (1.2)} Let  $b  \in {\bf Z}_p^*.$ Then we have 
 $$\alpha_p(1_{2k},b)= \beta_p(1_{2k},b)=1- p^{-2k},$$
 and 
 $$\alpha_p(1_{2k-1},b)= \beta_p(1_{2k-1},b)=1+ p^{-2k+1}.$$
\noindent
{\rm (2)} Suppose that $K_p$ is ramified over ${\bf Q}_p.$  \\
{\rm (2.1)} Let  $B \in  {\rm Her}_{m,*}({\mathcal O}_p)$ with $m \le 2.$
 Then we have
 $$\beta_p(\Theta_{2k},\pi^{i_p}B)=\prod_{i=0}^{m-1} (1-p^{-2k+2i}).$$
{\rm (2.2)}  Let $B=\smallmattwo(0;\varpi;\bar \varpi;0).$ Then we have
  $$\alpha_p(\Theta_{2k}, B)=\beta_p(\Theta_{2k}, B)= 1-p^{-2k}.$$
\end{lems}
\begin{proof}
(1) Put $B=(b).$ Let $p \not=2.$ Then we have $K_p={\bf Q}_p(\sqrt{\varepsilon})$ with $\varepsilon \in {\bf Z}_p^*$ such that $(\varepsilon,p)_p=-1.$ Then we have
$$\#{\mathcal B}_a(1_{2k},B)=\#\{(x_i) \in M_{4k,1}({\bf Z}_p)/p^a M_{4k,1}({\bf Z}_p) \ |  \  (x_i) \not\equiv 0 \ {\rm mod} \ p,$$
$$ \sum_{i=1}^{2k} (x_{2i-1}^2-\varepsilon x_{2i}^2) \equiv pb \ {\rm mod} \ p^a  \}.$$
Let $p=2.$ Then we have $K_2={\bf Q}_2(\sqrt{-3})$ and 
$$\#{\mathcal B}_a(1_{2k},B)=\#\{(x_i) \in M_{4k,1}({\bf Z}_2)/2^a M_{4k,1}({\bf Z}_2) \ |  \  (x_i) \not\equiv 0 \ {\rm mod} \ 2,$$
$$\sum_{i=1}^{2k} (x_{2i-1}^2+x_{2i-1}x_{2i}+ x_{2i}^2) \equiv 2b \ {\rm mod} \ 2^a \}.$$
In any case, by Lemma 9 of \cite{Ki1}, we have 
$$\#{\mathcal B}_a(1_{2k},B)=p^{(4k-1)a}(1-p^{-2k})(1+p^{-2k+1}).$$
This proves the assertion (1.1). Similarly the assertion (1.2) holds.

(2) First let $m=1,$ and put $B=(b)$ with $b \in 2{\bf Z}_p.$ Then $2^{-1} b \in {\bf Z}_p.$ 
Let $p \not=2,$ or $p=2$ and $f_2=3.$ Then we have $K_p={\bf Q}_p(\varpi)$ with $\varpi$ a prime element of
$K_p$ such that $\bar \varpi =-\varpi.$ Then an element ${\bf x}=(x_{2i-1}+\varpi x_{2i})_{ 1 \le i \le  2k}$ of $M_{2k,1}({\mathcal O}_p)/p^a M_{2k,1}({\mathcal O}_p)$ is primitive if and only if $(x_{2i-1})_{1 \le i \le 2k} \not\equiv  0 \ {\rm mod} \ p.$ Moreover we have 
$$\Theta_{2k}[{\bf x}]=2\sum_{1 \le i \le 2k} (x_{2i}x_{2i+1}-x_{2i-1}x_{2i+2})\pi.$$ Hence we have 
$$\#{\mathcal B}_a(1_{2k},B)=\#\{(x_i) \in M_{4k,1}({\bf Z}_p)/p^a M_{4k,1}({\bf Z}_p) \ |  \   (x_{2i-1})_{1 \le i \le 2k}  \not\equiv 0 \ {\rm mod} \ p$$
$$ \sum_{i=1}^{2k} (x_{2i}x_{2i+1}-x_{2i-1}x_{2i+2})  \equiv 2^{-1} b \ {\rm mod} \ p^a \}.$$
Let $p=2$ and $f_2=2.$ Then we have  $K_2={\bf Q}_2(\varpi)$ with $\varpi$ a prime element of
$K_2$ such that $\eta:=2^{-1}(\varpi+ \bar \varpi)  \in {\bf Z}_2^*.$
Then we have 
$$\#{\mathcal B}_a(1_{2k},B)=\#\{(x_i) \in M_{4k,1}({\bf Z}_2)/2^a M_{4k,1}({\bf Z}_2)
 \ |  \ (x_{2i-1})_{1 \le i \le 2k}  \not\equiv 0 \ {\rm mod} \ 2, $$
$$ \sum_{i=1}^{2k}\{\eta (x_{2i}x_{2i+1}+x_{2i-1}x_{2i+2}) + x_{2i-1}x_{2i+1}+\pi x_{2i}x_{2i+2}\}  \equiv 2^{-1}b \ {\rm mod} \ 2^a \}.$$
Thus, in any case,  by a simple computation we have 
$$\#{\mathcal B}_a(1_{2k},B)=p^{(2k-1)a}(p^{2ka}-p^{2k(a-1)}). $$
Thus the assertion (2.1) has been proved   for $m=1.$ Next let $\pi^{i_p}B =(b_{ij})_{1 \le i,j \le 2} \in {\rm Her}_{2,*}({\mathcal O}_p).$
Let $M=M_{\Theta_{2k}}, N_1=M_{\pi^{i_p}b_{11}},$  and $N=M_{B}.$ Let $a$ be a  sufficiently large positive integer.  
For each  $\phi_1 \in {\mathcal B}'_a(N_1,M),$ put
$${\mathcal B}'_a(N,M;\phi_1)=\{\phi \in {\mathcal B}'_a(N,M) \ | \ \phi |_{N_1}=\phi_1 \}.$$
Let $N={\mathcal O}_p v_1 \oplus {\mathcal O}_p v_2,$ and put $u_1=\phi_1(v_1).$ Then we can take elements
$u_2,\cdots,u_{2k} \in M$ such that
$$M={\mathcal O}_p u_1 \oplus {\mathcal O}_p u_2 \oplus \cdots \oplus {\mathcal O}_p u_{2k}$$
and
$$(u_1,u_2)=\varpi, (u_2,u_2)=0, (u_i,u_j)=0 \ {\rm for} \ i=1,2, j=3,\cdots, 2k, \ {\rm and} \ (u_i,u_j)_{3 \le i,j \le 2k}=\Theta_{2k-2}.$$
Then by the same argument as in the proof of Lemma 4.1.8, we can prove that
$${\mathcal B}'_a(N,M;\phi_1)=\{(x_i)_{1 \le i \le 2k-1}  \in M_{2k-1,1}({\mathcal O}_p)/p^a M_{2k-1,1}({\mathcal O}_p) \ | \ (x_i)_{2 \le i \le 2k-2} \ \not\equiv \ 0 \ {\rm mod} \ \varpi, $$
$$-x_1\bar x_1 b_{11} -x_1b_{12}-\bar x_1 \bar b_{12} +\Theta_{2k-2}[(x_i)_{2 \le i \le 2k-2}]\equiv b_{22} \ {\rm mod} \ p^a \}.$$
Hence by the assertion for $m=1,$ we have 
$$\beta_p(\Theta_{2k},B)=\beta_p(\Theta_{2k},b_{11})p^{-a}\sum_{x_1 \in {\mathcal O}_p/\varpi^a {\mathcal O}_p} \beta_p(\Theta_{2k-2}, b_{22}+b_{11}x_1 \bar x_1+x_1b_{12}+\bar x_1 \bar b_{12})$$
$$=(1-p^{-2k})(1-p^{-2k+2}).$$
Thus the assertion (2.1) has been proved   for $m=2.$
The assertion (2.2)  can be proved by using the same argument as above.

\end{proof}

\begin{lems}
Let $k$ and $m$ be integers with $k \ge m.$

\noindent
{\rm (1)} Suppose that $K_p$ is unramified over ${\bf Q}_p.$ 
  Let $A \in {\rm Her}_l({\mathcal O}_p)$ and   $B \in {\rm Her}_m({\mathcal O}_p).$ Then  we have
  $$\beta_p(pA \bot 1_{2k},pB)=\beta_p(1_{2k}, pB)=  \prod_{i=0}^{2m-1} (1-(-1)^i p^{-2k+i})$$
  
\noindent
{\rm (2)} Let $K_p={\bf Q }_p \oplus {\bf Q}_p.$  Let $l$ be an integer.
  Let  $B \in {\rm Her}_{m}({\mathcal O}_p).$ Then  we have
  $$\beta_p(1_{2k},pB )=  \prod_{i=0}^{2m-1} (1-p^{-2k+i})$$\\

\noindent
{\rm (3)} Suppose that $K_p$ is ramified over ${\bf Q}_p.$  Let  $A \in {\rm Her}_{l,*}({\mathcal O}_p)$ and $B \in  {\rm Her}_{m,*}({\mathcal O}_p).$ 
 Then we have
 $$\beta_p(\pi^{i_p}A \bot \Theta_{2k},\pi^{i_p}B)=\beta_p(\Theta_{2k},\pi^{i_p}B)= \prod_{i=0}^{m-1} (1-p^{-2k+2i}).$$
\end{lems}
 \begin{proof}
(1) Suppose that $K_p$ is unramified over ${\bf Q}_p.$ We prove the assertion by induction on $m.$
Let $\deg B=1,$ and $a$ be a sufficiently large integer. Then, by Lemma 4.1.9,  we have
$$\beta_p(pA \bot 1_{2k},pB)=p^{-al} \sum_{{\bf x} \in M_{l1}({\mathcal O}_p)/p^a M_{l1}({\mathcal O}_p)} \beta_p( 1_{2k},pB-pA[{\bf x}])$$
$$=(1-p^{-2k})(1+p^{-2k+1}).$$
This proves the assertion for $m=1.$ Let $m >1$ and suppose that the assertion holds for $m-1.$
Then $B$ can be expressed as
$B \sim_{GL_m({\mathcal O}_p)}  B_1 \bot B_2$ with $ B_1 \in {\rm Her}_1({\mathcal O}_p)$ and $B_2 \in  {\rm Her}_{m-1}({\mathcal O}_p).$
Then by Lemma 4.1.8, we have 
$$\beta_p(pA \bot 1_{2k}, pB_1 \bot pB_2)=\beta_p(pA \bot 1_{2k},pB_1)\beta_p(pA  \bot (-pB_1) \bot  1_{2k-2},  pB_2).$$
Thus the assertion holds by the induction assumption. 

(2) Suppose that  $K_p={\bf Q}_p \oplus {\bf Q}_p.$  Then  we easily see that we have 
$$\beta_p(1_{2k},pB)=p^{(-4km+m^2)} \# {\mathcal B}_1(1_{2k},O_m).$$ 
We have 
$$ {\mathcal B}_1(1_{2k},O_m)$$
$$=\{(X,Y) \in M_{2k,l}({\bf Z}_p)/pM_{2k,l}({\bf Z}_p)  \oplus M_{2k,l}({\bf Z}_p)/pM_{2k,l}({\bf Z}_p) \ | $$
$$  {}^t YX \equiv O_m \ {\rm mod} \  p M_m({\bf Z}_p)  \ {\rm and} \ {\rm rank}_{{\bf Z}_p/p{\bf Z}_p} X={\rm rank}_{{\bf Z}_p/p{\bf Z}_p} Y=m\}.$$
For each $X \in M_{2k,l}({\bf Z}_p)/pM_{2k,l}({\bf Z}_p)$ such that  ${\rm rank}_{{\bf Z}_p/p{\bf Z}_p} X=m,$
put 
$$\# {\mathcal B}_1(1_{2k},O_m;X)$$
$$=\{Y \in M_{2k,l}({\bf Z}_p)/pM_{2k,l}({\bf Z}_p) \ | \   {}^t YX \equiv O_m \ {\rm mod} \  p M_m({\bf Z}_p)  \ {\rm and} \ {\rm rank}_{{\bf Z}_p/p{\bf Z}_p} Y=m\}.$$
By a simple computation we have
$$\#\{X \in M_{2k,l}({\bf Z}_p)/pM_{2k,l}({\bf Z}_p) \ | \  {\rm rank}_{{\bf Z}_p/p{\bf Z}_p} X=m\}=
\prod_{i=0}^{m-1} (p^{2k}-p^{i}),$$
and 
$$\# {\mathcal B}_1(1_{2k},O_m;X)=\prod_{i=0}^{m-1} (p^{2k-m}-p^{i}).$$
This proves the assertion.

{\rm (3)} Suppose that $K_p$ is ramified over  ${\bf Q}_p.$ We prove the assertion by induction on $m.$
Let $\deg B=1,$ and $a$ be a sufficiently large integer. Then, by Lemma 4.1.9,  we have
$$\beta_p(\pi^{i_p} A \bot \Theta_{2k},\pi^{i_p}B)=p^{-al} \sum_{{\bf x} \in M_{l1}({\mathcal O}_p)/p^a M_{l1}({\mathcal O}_p)} \beta_p( \Theta_{2k},\pi^{i_p}B-\pi^{i_p}A[{\bf x}])
=1-p^{-2k}.$$
Let $\deg B=2.$ Then by Lemma 4.1.9, we have 
$$\beta_p(\pi^{i_p} A \bot \Theta_{2k},\pi^{i_p}B)=p^{-2la} \sum_{{\bf x} \in M_{l2}({\mathcal O}_p)/p^a M_{l2}({\mathcal O}_p)} \beta_p( \Theta_{2k},\pi^{i_p}B-\pi^{i_p}A[{\bf x}])$$
$$=(1-p^{-2k})(1-p^{-2k+2}).$$
Let $m \ge 3.$ Then $B$ can be expressed as 
$B \sim_{GL_m({\mathcal O}_p)} B_1 \bot B_2$ with $\deg B_1 \le 2.$ Then the assertion for $m$ holds by Lemma 4.1.8, the induction hypothesis, and Lemma 4.1.9. 

\end{proof}
  \begin{lems}
{\rm (1)} Suppose that $K_p$ is unramified over ${\bf Q}_p.$ Let $l$ and $m$ be an integers with $l \ge m.$
   Then  we have
  $$\alpha_p(1_l,1_m)=\beta_p(1_l,1_m)= \prod_{i=0}^{m-1} (1-(-p)^{-l+i})$$
  
\noindent
{\rm (2)} Let $K_p={\bf Q }_p \oplus {\bf Q}_p.$  Let $l$ and $m$ be  integers with $l \ge m.$
  Then  we have
  $$\alpha_p(1_l,1_m )= \beta_p(1_l,1_m )= \prod_{i=0}^{m-1} (1-p^{-l+i})$$\\

\noindent
{\rm (3)} Suppose that $K_p$ is ramified over ${\bf Q}_p.$ Let $k$ and $m$ be even integers with $k \ge m.$
 Then we have
 $$\alpha_p(\Theta_{2k},\Theta_{2m})=\beta_p(\Theta_{2k},\Theta_{2m})= \prod_{i=0}^{m-1} (1-p^{-2k+2i}).$$
\end{lems}
\begin{proof}
In any case, we easily see that the local density coincides with the primitive local density. 
Suppose that $K_p$ is unramified over ${\bf Q}_p.$ Then, 
by Lemma 4.1.7, we have
$$\alpha_p(1_l,1_m)=\alpha_p(1_l,1)\alpha_p(1_{l-1},1_{m-1}).$$
We easily see that we have 
$$\alpha_p(1_l,1)=1-(-1)^l p^{-l}.$$
This proves the assertion (1). 
Suppose that $K_p$ is ramified over ${\bf Q}_p.$ Then by Lemma 4.1.7, we have
$$\alpha_p(\Theta_{2k}, \Theta_m)=\alpha_p(\Theta_{2k},\Theta_2)\alpha_p(\Theta_{2k-2},\Theta_{2m-2}).$$
Moreover by Lemma 4.1.9, we have 
$$\alpha_p(\Theta_{2k},\Theta_2)=1-p^{-2k}.$$
This proves the assertion (3).
Suppose that $K_p={\bf Q}_p \oplus {\bf Q}_p.$ Then the assertion can be proved similarly to (2) of Lemma 4.1.10.

 \end{proof} 

\subsection{Primitive densities}
  
  \noindent
  
  { }
  
  \bigskip
 
For an element $T \in {\widetilde{\rm Her}}_{m}({\mathcal O}_p),$  we define a polynomial 
 $G_p(T,X)$ in $X$ by
$$G_p(T,X)=\sum_{i=0}^{m} \sum_{W \in GL_{m}({\mathcal O}_p) \backslash {\mathcal D}_{m,i}} (Xp^{m})^{\nu(\det W)}\varPi_p(W)F_p^{(0)}(T[W^{-1}],X).$$

 \bigskip

\begin{lems}
{\rm (1)} Suppose that $K_p$ is unramified over ${\bf Q}_p.$ 
  Let  $B_1 \in {\rm Her}_{m-n_0}({\mathcal O}_p).$ Then  we have
  $$\alpha_p(1_{n_0} \bot pB_1)=  \prod_{i=1}^{n_0} (1-(-p)^{-i}) \alpha_p(pB_1)$$
  
\noindent
{\rm (2)} Let $K_p={\bf Q }_p \oplus {\bf Q}_p.$  
  Let  $B_1 \in {\rm Her}_{m-n_0}({\mathcal O}_p).$ Then  we have
  $$\alpha_p(1_{n_0} \bot pB_1 )=  \prod_{i=1}^{n_0} (1-p^{-i})\alpha_p(pB_1) $$\\

\noindent
{\rm (3)} Suppose that $K_p$ is ramified over ${\bf Q}_p.$ Let $n_0$ be even. Let  $B_1 \in  {\rm Her}_{m-n_0,*}({\mathcal O}_p).$ 
 Then we have
 $$\alpha_p(\Theta_{n_0} \bot \pi^{i_p}B_1)= \prod_{i=1}^{n_0/2} (1-p^{-2i})\alpha_p(\pi^{i_p}B_1).$$
\end{lems}

\begin{proof} Suppose that $K_p$ is unramified over ${\bf Q}_p.$ By Lemma 4.1.7, we have
$$\alpha_p(1_{n_0} \bot pB_1)=\alpha_p(1_{n_0} \bot pB_1,1_{n_0})\alpha_p(pB_1).$$
By using the same argument as in the proof of Lemma 4.1.10, we can prove that we have 
$$\alpha_p(1_{n_0} \bot pB_1,1_{n_0})=\alpha_p(1_{n_0}),$$
 and hence by Lemma 4.1.11, we have
$$\alpha_p(1_{n_0} \bot pB_1)=\prod_{i=1}^{n_0} (1-(-p)^{-i})\alpha_p(pB_1).$$
This proves the assertion (1).
Similarly the assertions (2) and (3) can be  proved. 
\end{proof}

 \begin{lems}
 Let $m$ be a positive integer and $r$ a non-negative integer such that $r \le m.$
 
 \noindent
{\rm (1)} Suppose that $K_p$ is unramified over ${\bf Q}_p.$  Let $T=1_{m-r} \bot pB_1$  with $B_1 \in {\rm Her}_{r}({\mathcal O}_p).$ Then
  $$\beta_p(1_{2k},T)= \prod_{i=0}^{m+r-1}(1-p^{-2k+i}(-1)^i).$$

\noindent
 {\rm (2)} Suppose that $K_p={\bf Q}_p \oplus {\bf Q}_p.$  Let $T=1_{m-r} \bot pB_1$  with $B_1 \in {\rm Her}_{r}({\mathcal O}_p).$ Then
  $$\beta_p(1_{2k},T)= \prod_{i=0}^{m+r-1}(1-p^{-2k+i}).$$

 \noindent
 {\rm (3)} Suppose that $K_p$ is ramified over ${\bf Q}_p$  and that  $m-r$ is even. Let $T=\Theta_{m-r} \bot \pi^{i_p}B_1$  with $B_1 \in {\rm Her}_{r,*}({\mathcal O}_p).$ Then
  $$\beta_p(\Theta_{2k},T)= \prod_{i=0}^{(m+r-2)/2}(1-p^{-2k+2i}).$$
  \end{lems}

\begin{proof}
Suppose that $K_p$ is unramified over ${\bf Q}_p.$ By Lemma 4.1.8, we have
$$\beta_p(1_{2k},T)=\beta_p(1_{2k},pB_1)\beta_p( (-pB_1) \bot 1_{2k-2r},1_{m-r}).$$
By using the same argument as in the proof of Lemma 4.1.11, we can prove that we have
$\beta_p( (-pB_1) \bot 1_{2k-2r},1_{m-r})=\beta_p(1_{2k-2r},1_{m-r}).$ Hence the assertion follows from  Lemmas 4.1.10 and 4.1.11.
 Similarly the assertions (2) and (3) can be proved.

\end{proof}

\bigskip
 \begin{xcor}
   {\rm (1)} Suppose that $K_p$ is unramified over ${\bf Q}_p$ or $K_p={\bf Q}_p \oplus {\bf Q}_p.$ Let $T=1_{m-r} \bot pB_1$  with $B_1 \in {\rm Her}_r({\mathcal O}_p).$ Then we have
  $$G_p(T,Y)= \prod_{i=0}^{r-1}(1-(\xi_p p)^{m+i}Y).$$
 
 {\rm (2)} Suppose that $K_p$ is ramified over ${\bf Q}_p$ and that  $m-r$ is even. Let $T=\Theta_{m-r} \bot \pi^{i_p}B_1$  with $B_1 \in {\rm Her}_{r,*}({\mathcal O}_p).$  Then
   $$G_p(T,Y)=\prod_{i=0}^{[(r-2)/2]}(1-p^{2i+2[(m+1)/2]}Y).$$
   \end{xcor}

\begin{proof} Let $k$ be a positive integer such that $k \ge m.$ Put $\Xi_{2k}=\Theta_{2k}$ or $1_{2k}$ according as $K_p$ is ramified over ${\bf Q}_p$ or not. 
Then it follows from Lemma 14.8 of \cite{Sh1} that for $B \in \widetilde {Her}_m({\mathcal O}_p)^{\times}$ we have
$$b_p(p^{-e_p}B,2k)=\alpha_p(\Xi_{2k},B).$$
Hence, by the definition of $G_p(T,X)$ and Corollary to Lemma 4.1.2, we have 
$$\beta_p(\Xi_{2k},T)=G_p(T,p^{-2k})\prod_{i=0}^{[(m-1)/2]}(1-p^{2i-2k})\prod_{i=1}^{[m/2]} (1-\xi_p p^{2i-1-2k}). $$ 
Suppose that $K_p$ is unramified over ${\bf Q}_p$ or $K_p={\bf Q}_p \oplus {\bf Q}_p.$ Then by Lemma 4.2.2, we have
$$G_p(T,p^{-2k})= \prod_{i=0}^{r-1}(1-(\xi_p p)^{m+i}p^{-2k}).$$
This equality holds for infinitely many positive integer $k,$ and the both hand sides of it are polynomials in $p^{-2k}.$  
Thus the assertion (1) holds. Similarly the assertion (2) holds.   
\end{proof}  

\bigskip
 
\begin{lems}
Let $B \in {\widetilde{\rm Her}}_{m}({\mathcal O}_p).$ Then we have
$$F_p^{(0)}(B,X)= \sum_{W \in  GL_{m}({\mathcal O}_p) \backslash \widetilde \Omega(B)} G_p(B[W^{-1}],X)(p^mX)^{\nu(\det W)}.$$
\end{lems}

\begin{proof}
Let $k$ be a positive integer such that $k \ge m.$
By Lemma 4.1.2, we have
$$\alpha_p(\Xi_{2k},B)=\sum_{W \in  GL_{m}({\mathcal O}_p) \backslash \widetilde \Omega(B)} \beta_p(\Xi_{2k},B[W^{-1}])p^{(-2k+m)\nu(\det W)}.$$
Then the assertion can be proved by using the same argument as in the proof of Corollary to Lemma 4.2.2.
\end{proof}

 \begin{xcor}
Let $B \in {\widetilde{\rm Her}}_{m}({\mathcal O}_p).$ Then we have
$$\widetilde F^{(0)}(B,X)=X^{e_pm-f_p[m/2]}\sum_{B' \in  {\widetilde{\rm Her}}_{m}({\mathcal O}_p) / GL_{m}({\mathcal O}_p)  } X^{-{\rm ord}(\det B')}{\alpha_p(B',B) \over \alpha_p(B')}$$$$ \times  G_p(B',p^{-m}X^2)X^{{\rm ord}(\det B)-{\rm ord}(\det B')}.$$
\end{xcor}
\begin{proof}  We have 
$$\widetilde F^{(0)}(B,X)=X^{e_pm-f_p[m/2]}X^{-{\rm ord}(\det B)}F^{(0)}(B,p^{-m}X^2) $$
$$=X^{e_pm-f_p[m/2]}\sum_{W \in  GL_{m}({\mathcal O}_p) \backslash \widetilde \Omega(B)} X^{-{\rm ord}(\det B)}G_p(B[W^{-1}],p^{-m}X^2)(X^2)^{\nu(\det W)}$$
$$=X^{e_pm-f_p[m/2]}$$
$$ \times \sum_{B' \in   {\widetilde{\rm Her}}_{m}({\mathcal O}_p)/GL_{m}({\mathcal O}_p)}\sum_{W \in  GL_{m}({\mathcal O}_p) \backslash \widetilde \Omega(B',B)} X^{-{\rm ord}(\det B)}G_p(B',p^{-m} X^2)(X^2)^{\nu(\det W)}$$
$$=X^{e_pm-f_p[m/2]}\sum_{B' \in   {\widetilde{\rm Her}}_{m}({\mathcal O}_p)/GL_{m}({\mathcal O}_p)  }X^{-{\rm ord}(\det B')}  \#(GL_{m}({\mathcal O}_p) \backslash \widetilde {\Omega}(B',B))$$
$$ \times  G_p(B',p^{-m}X^2)X^{{\rm ord}(\det B)-{\rm ord}(\det B')}.$$ 
Thus the assertion follows from (2) of Lemma 4.1.3.
\end{proof}

\bigskip

Let
$$\widetilde{\mathcal F}_{m,p}(d_0)=\bigcup_{i=0}^{\infty} (\widetilde {\rm Her}_m(\pi^id_0N_{K_p/{\bf Q}_p}({\mathcal O}_p^*),{\mathcal O}_p),$$
and  
$${\mathcal F}_{m,p,*}(d_0)=\widetilde{\mathcal F}_{m,p}(d_0) \cap {\rm Her}_{m,*}({\mathcal O}_p).$$
 
First suppose that $K_p$ is unramified over ${\bf Q}_p$ or $K_p={\bf Q}_p \oplus {\bf Q}_p.$ Let  $H_{m}$  be a function on  ${\rm Her}_m({\mathcal O}_p)^{\times}$  satisfying the following condition: 

\bigskip

$H_{m}(1_{m-r} \bot pB)=H_{r}(pB)$  for any  $B \in {\rm Her}_{r}({\mathcal O}_p).$

\bigskip

\noindent
Let $d_0 \in {\bf Z}_p^*.$  Then we put
$$Q(d_0,H_{m},r,t)= \sum_{B \in p^{-1}\widetilde{\mathcal F}_{r,p}(d_0) \cap {\rm Her}_{r} ({\mathcal O}_p)}{H_{m}(1_{m-r} \bot pB) \over \alpha_p(1_{m-r} \bot pB)}t^{{\rm ord}(\det (pB))}.$$
Next suppose that $K_p$ is ramified over ${\bf Q}_p.$  
Let  $H_{m}$  be a function on  ${\rm Her}_m({\mathcal O}_p)^{\times}$  satisfying the following condition: 

\bigskip

$H_{m}(\Theta_{m-r} \bot \pi^{i_p}B)=H_{r}(\pi^{i_p}B)$  for any  $B \in {\rm Her}_{r,*}({\mathcal O}_p)$ if $m-r$ is even. 

\bigskip

\noindent
Let $d_0 \in {\bf Z}_p^*$ and $m-r$ be even.   Then we put
$$Q(d_0,H_{m},r,t)=\sum_{B \in \pi^{-i_p}\widetilde {\mathcal F}_{r,p}(d_0) \cap {\rm Her}_{r,*} ({\mathcal O}_p)}{H_{m}(\Theta_{m-r} \bot \pi^{i_p}B)  \over \alpha_p(\Theta_{m-r} \bot \pi^{i_p}B)}t^{{\rm ord}(\det (\pi^{i_p}B))}.$$
Then by Lemma 4.2.1 we easily obtain the following. 

\bigskip

   
\begin{props}   
  
  \noindent 
  {\rm (1)} Suppose that $K_p$ is unramified over ${\bf Q}_p$ or $K_p={\bf Q}_p \oplus {\bf Q}_p.$ Then for any $d_0 \in {\bf Z}_p^*$ and a non-negative integer $r$ we have 
$$Q(d_0,H_{m},r,t)={Q(d_0,H_{r},r,t) \over \phi_{m-r}(\xi_p p^{-1})}.$$

  \noindent 
  {\rm (2)} Suppose that $K_p$ is ramified over ${\bf Q}_p.$  Then for any $d_0 \in {\bf Z}_p^*$ and a non-negative integer $r$ such that $m-r$ is even, we have 
$$Q(d_0,H_{m},r,t)={Q(d_0,H_{r},r,t) \over \phi_{(m-r)/2}(p^{-2})}.$$
\end{props}

\bigskip


 \subsection{Explicit formulas of formal power series of Koecher-Maass type}
 
 { }
 \noindent
 { }
 
 \bigskip
 
 In this section we give an explicit formula for $P_m(d_0,X,t).$  

\begin{thms}
 Let $m$ be even, and $d_0 \in {\bf Z}_{p}^*.$ 
 
\noindent 
 {\rm (1)} Suppose that $K_p$ is unramified over ${\bf Q}_p.$   Then 
 $$P_m(d_0,X,t)={ 1 \over \phi_{m}(-p^{-1})\prod_{i=1}^{m} (1-t (-p)^{-i}X)(1+t(-p)^{-i}X^{-1})  }.$$

\noindent  
  {\rm (2)} Suppose that $K_p={\bf Q}_p \oplus {\bf Q}_p.$ Then 
  $$P_m(d_0,X,t)={ 1 \over \phi_{m}(p^{-1})\prod_{i=1}^{m} (1-t p^{-i}X)(1-tp^{-i}X^{-1})  }.$$

\noindent 
 {\rm (3)} Suppose that $K_p$ is ramified over ${\bf Q}_p.$   Then 
 $$P_m(d_0,X,t)={ t^{mi_p/2} \over 2\phi_{m/2}(p^{-2})}$$
$$\times \left\{{1 \over \prod_{i=1}^{m/2} (1-t p^{-2i+1}X^{-1})(1-tp^{-2i}X)} + {\chi_{K_p}((-1)^{m/2}d_0)  \over \prod_{i=1}^{m/2} (1-t p^{-2i}X^{-1})(1-tp^{-2i+1}X)}\right \}.$$

\end{thms}
   
\begin{thms}
 Let $m$ be odd, and $d_0 \in {\bf Z}_{p}^*.$ 
 
\noindent 
 {\rm (1)} Suppose that $K_p$ is unramified over ${\bf Q}_p.$   Then 
 $$P_m(d_0,X,t)={ 1 \over \phi_{m}(-p^{-1})\prod_{i=1}^{m} (1+t (-p)^{-i}X)(1+t(-p)^{-i}X^{-1})  }.$$
 
\noindent  
  {\rm (2)} Suppose that $K_p={\bf Q}_p \oplus {\bf Q}_p.$ Then 
  $$P_m(d_0,X,t)={ 1 \over \phi_{m}(p^{-1})\prod_{i=1}^{m} (1-t p^{-i}X)(1-tp^{-i}X^{-1})  }.$$

\noindent 
 {\rm (3)} Suppose that $K_p$ is ramified over ${\bf Q}_p.$   Then 
 $$P_m(d_0,X,t)={ t^{(m+1)i_p/2+\delta_{2p}} \over 2\phi_{(m-1)/2}(p^{-2})\prod_{i=1}^{(m+1)/2} (1-t p^{-2i+1}X)(1-tp^{-2i+1}X^{-1})  }.$$
\end{thms}  
   
To prove Theorems 4.3.1 and 4.3.2, put 
$$K_{m}(d_0,X,t)=\sum_{B' \in \widetilde{\mathcal F}_{r,p}(d_0)}{G_p(B',p^{-m}X^2) \over \alpha_p(B')}(tX^{-1})^{{\rm ord}(\det B')}.$$

\begin{props}
Let $m$ and $d_0$ be as above.  Then we have
$$P_{m}(d_0,X,t)=X^{me_p-[m/2]f_p} K_{m}(d_0,X,t)$$
$$\times \listthree({\prod_{i=1}^{m}(1-t^2X^2p^{2i-2-2m})^{-1}};{ \ {\rm if} \ K_p/{\bf Q}_p \ {\rm is \ unramified }};{\prod_{i=1}^{m}(1-tXp^{i-1-m})^{-2}};{ \ {\rm if} \ K_p={\bf Q}_p \oplus {\bf Q}_p };{\prod_{i=1}^{m}(1-tXp^{i-1-m})^{-1}};{ \ {\rm if} \ K_p/{\bf Q}_p \ {\rm is \ ramified }.})$$
\end{props}

\begin{proof} We note that $B'$ belongs to ${\widetilde{\rm Her}}_{m,p}(d_0)$ if $B$ belongs to ${\widetilde{\rm Her}}_{m-l,p}(d_0)$ and $\alpha_p(B',B) \not=0.$ Hence  by Corollary to Lemma 4.2.3 we have
$$P_{m}(d_0,X,t)$$
$$=X^{me_p-[m/2]f_p}  \sum_{B \in \widetilde{\mathcal F}_{m,p}(d_0)}{1 \over \alpha_p(B)}\sum_{B'} {G_p(B',p^{-m}X^2) X^{-{\rm ord}(\det B')}\alpha_p(B',B) \over \alpha_p(B')}$$
$$\times X^{{\rm ord}(\det B)-{\rm ord}(\det B')}t^{{\rm ord}(\det B)}$$
$$=X^{me_p-[m/2]f_p} \sum_{B' \in \widetilde{\mathcal F}_{m,p}(d_0)}{G_p(B',p^{-m}X^2) \over \alpha_p(B')}(tX^{-1})^{{\rm ord}(\det B')}$$
$$\times \sum_{B \in \widetilde{\mathcal F}_{m,p}(d_0)} {\alpha_p(B',B) \over \alpha_p(B)} (tX)^{{\rm ord}(\det B)-{\rm ord}(\det B')}.$$
Hence  by using the same argument as in the proof of [\cite{B-S}, Theorem 5], and by (1) of Lemma 4.1.3, we have 
$$\sum_{B \in \widetilde{\mathcal F}_{m,p}(d_0)}{\alpha_p(B',B) \over \alpha_p(B)} (tX)^{{\rm ord}(\det B)-{\rm ord}(\det B')} $$
$$=\sum_{W \in M_m({\mathcal O}_p)^{\times}/GL_m({\mathcal O}_p)} (tXp^{-m})^{\nu(\det W)}$$
$$=\listthree({\prod_{i=1}^{m}(1-t^2X^2p^{2i-2-2m})^{-1}};{ \ {\rm if} \ K_p/{\bf Q}_p \ {\rm is \ unramified }};{\prod_{i=1}^{m}(1-tXp^{i-1-m})^{-2}};{ \ {\rm if} \ K_p={\bf Q}_p \oplus {\bf Q}_p };{\prod_{i=1}^{m}(1-tXp^{i-1-m})^{-1}};{ \ {\rm if} \ K_p/{\bf Q}_p \ {\rm is \ ramified }.})$$
Thus the assertion holds.
\end{proof}

\bigskip

In order to prove Theorems 4.3.1 and 4.3.2,  we introduce some  notation.   For a positive integer $r$ and $d_0 \in {\bf Z}_p^{\times}$ let
\[
\zeta_{m}(d_{0},t)=  \sum_{T \in {\mathcal F}_{m,p,*}(d_0)}
\frac{1}{\alpha_{p}(T)}t^{{\rm ord}(\det T)}.
\]
We make the convention that 
 $\zeta_{0}(d_0,t) = 1$  or $0$ according as  $d_0 \in {\bf Z}_p^*$ or not. 
To obtain an explicit formula of $\zeta_{m}(d_{0},t)$ let $Z_m(u,d)$ be the integral  defined as 
$$Z_{m,*}(u,d)=\int_{{\mathcal F}_{m,p,*}(d_0)} |\det x|^{s-m} |dx|,$$
 where $u=p^{-s},$ and $|dx|$ is the measure on ${\rm Her}_m(K_p)$ so that the volume of ${\rm Her}_m({\mathcal O}_p)$ is $1.$ Then by Theorem 4.2 of \cite{Sa} we obtain:
 
 \bigskip
 
 \begin{props}
 Let $d_0 \in {\bf Z}_p^*.$ 
 
 \noindent
  {\rm (1)} Suppose that $K_p$ is unramified over ${\bf Q}_p.$ Then
 $$Z_{m,*}(u,d_0)={(p^{-1},p^{-2})_{[(m+1)/2]}(-p^{-2},p^{-2})_{[m/2]} \over \prod_{i=1}^{m}(1-(-1)^{m+i}p^{i-1}u)}.$$
   
    {\rm (2)} Suppose that $K_p={\bf Q}_p \oplus {\bf Q}_p.$    Then 
 $$Z_{m,*}(u,d_0)={\phi_m(p^{-1}) \over \prod_{i=1}^{m}(1-p^{i-1}u)}.$$
 {\rm (3)} Suppose that $K_p$ is ramified over ${\bf Q}_p.$  \\
 {\rm (3.1)} Let $p\not=2.$ Then
 $$Z_{m,*}(u,d_0)={1 \over 2} (p^{-1},p^{-2})_{[(m+1)/2]}$$
 $$\times  \left\{\begin{array}{ll}
   { 1  \over  \prod_{i=1}^{(m+1)/2} (1-p^{2i-2}u)} & \ {\rm if} \ m \ {\rm is \ odd,} \\
 \Bigl( { 1 \over \prod_{i=1}^{m/2}(1-p^{2i-1}u)} +{ \chi_{K_p}((-1)^{m/2}d_0)p^{-m/2} \over  \prod_{i=1}^{m/2}(1-p^{2i-2}u)}\Bigr) & \ {\rm if} \ m \ {\rm is \ even.} 
 \end{array}
 \right.$$
 {\rm (3.2)} Let $p=2$ and $f_2=2.$ Then
 $$Z_{m,*}(u,d_0)={1 \over 2} (p^{-1},p^{-2})_{[(m+1)/2]}$$
 $$\times  \left\{\begin{array}{ll}
   { u^{(m+1)/2}  \over  \prod_{i=1}^{(m+1)/2} (1-p^{2i-2}u)} & \ {\rm if} \ m \ {\rm is \ odd,} \\
 u^{m/2}p^{-m/2} \Bigl({ 1 \over   \prod_{i=1}^{m/2}(1-p^{2i-1}u)} +{ \chi_{K_p}((-1)^{m/2}d_0)p^{-m/2} \over  \prod_{i=1}^{m/2}(1-p^{2i-2}u)}\Bigr) & \ {\rm if} \ m \ {\rm is \ even.} 
 \end{array}
 \right.$$
 {\rm (3.3)} Let $p=2$ and $f_2=3.$ Then
 $$Z_{m,*}(u,d_0)={1 \over 2} (p^{-1},p^{-2})_{[(m+1)/2]}$$
 $$\times  \left\{\begin{array}{ll}
   { u  \over  \prod_{i=1}^{(m+1)/2} (1-p^{2i-2}u)} & \ {\rm if} \ m \ {\rm is \ odd,} \\
 p^{-m} \Bigl({ 1 \over  \prod_{i=1}^{m/2}(1-p^{2i-1}u)} +{ \chi_{K_p}((-1)^{m/2}d_0)p^{-m/2} \over  \prod_{i=1}^{m/2}(1-p^{2i-2}u)}\Bigr) & \ {\rm if} \ m \ {\rm is \ even.} 
 \end{array}
 \right.$$
 \end{props}
 
 \begin{proof} First suppose that $K_p$ is unramified over ${\bf Q}_p,$ $K_p={\bf Q}_p \oplus {\bf Q}_p,$ or $K_p$ is ramified over ${\bf Q}_p$ and $p \not=2.$ Then $Z_{m,*}(u,d_0)$ coincides with $Z_m(u,d_0)$ in [\cite{Sa}, Theorem 4.2]. Hence the assertion follows from (1) and (2) and the former half of (3) of [loc. cit].  Next suppose that $p=2$ and $f_2=2.$ Then $Z_{m,*}(u,d_0)$ is not treated in  [loc. cit],  but we can prove the assertion (3.2) using the same argument as in the proof of the latter half of (3) of [loc. cit]. Similarly we can prove (3.3) by using the same argument as in the proof of the former half of (3) of [loc. cit].
 
\end{proof} 

\begin{xcor}
Let $d_0 \in {\bf Z}_p^*.$

\noindent
{\rm (1)} Suppose that $K_p$ is unramified over ${\bf Q}_p.$   Then 
$$\zeta_m(d_0,t)={1 \over \phi_m(-p^{-1})}{1 \over \prod_{i=1}^m (1+(-1)^ip^{-i}t)}.$$

\noindent
{\rm (2)} Suppose that $K_p={\bf Q}_p \oplus {\bf Q}_p.$    Then 
$$\zeta_m(d_0,t)={1 \over \phi_m(p^{-1})}{1 \over \prod_{i=1}^m (1-p^{-i}t)}.$$

\noindent
{\rm (3)} Suppose that $K_p$ is ramified over ${\bf Q}_p.$   

\noindent
{\rm (3.1)} Let $m$ be even. Then 
$$\zeta_m(d_0,t)={ p^{m(m+1)f_p/2-m^2\delta_{2,p}} {\kappa_p}(t) \over 2\phi_{m/2}(p^{-2})}$$$$ \times \{ {1   \over \prod_{i=1}^{m/2} (1-p^{-2i+1} t)} + {\chi_{K_p}((-1)^{m/2}d_0) p^{-i_pm/2} \over \prod_{i=1}^{m/2} (1-p^{-2i} t)}\},$$
where $i_p=0,$ or $1$ according as $p=2$ and $f_p=2$, or not, and 
$$\kappa_p(t)=\left\{\begin{array} {ll}
1    & \ {\rm if} \ p \not=2 \\
t^{m/2}p^{-m(m+1)/2} & \ {\rm if} \ p=2 \ {\rm and} \ f_2=2 \\
p^{-m}  & \ {\rm if} \ p=2 \ {\rm and} \ f_2=3 
\end{array}
\right.$$

\noindent
{\rm (3.2)} Let $m$ be odd.  Then
$$\zeta_m(d_0,t)={p^{m(m+1)f_p/2-m^2\delta_{2,p}} \kappa_p(t) \over 2\phi_{(m-1)/2}(p^{-2})}{1 \over \prod_{i=1}^{(m+1)/2} (1-p^{-2i+1}t)},$$
where 
$$\kappa_p(t)=\left\{\begin{array} {ll}
1    & \ {\rm if} \ p \not=2 \\
t^{(m+1)/2}p^{-m(m+1)/2} & \ {\rm if} \ p=2 \ {\rm and} \ f_2=2 \\
tp^{-m}   & \ {\rm if} \ p=2 \ {\rm and} \ f_2=3 
\end{array}
\right.$$

\end{xcor}
\begin{proof}
First suppose that $K_p$ is unramified over ${\bf Q}_p.$ Then by a simple computation  we have 
$$\zeta_m(d_0,t)={Z_{m,*}(p^{-m}t,d_0) \over \phi_m(p^{-2})}.$$
Next@suppose that $K_p={\bf Q}_p \oplus {\bf Q}_p.$ Then similarly to above 
$$\zeta_m(d_0,t)={Z_{m,*}(p^{-m}t,d_0) \over \phi_m(p^{-1})^2}.$$
Finally suppose that $K_p$ is ramified over ${\bf Q}_p.$
Then by a simple computation and Lemma 3.1
$$\zeta_m(d_0,t)={p^{m(m+1)f_p/2-m^2\delta_{2,p}}Z_{m,*}(p^{-m}t,d_0) \over \phi_m(p^{-1})}.$$
  Thus the assertions follow from Proposition 4.3.4.

\end{proof}

\begin{props}
Let $d_0 \in {\bf Z}_p^*.$ 

\noindent
{\rm (1)} Suppose that $K_p$ is unramified over ${\bf Q}_p.$   Then
$$K_{m}(d_0,X,t)$$
$$=\sum_{r=0}^{m}{p^{-r^2}(tX^{-1})^r \prod_{i=0}^{r-1}(1-(-1)^m(-p)^{i}X^2) \over \phi_{m-r}(-p^{-1})}\zeta_{r}(d_0,tX^{-1}).$$

\noindent
{\rm (2)} Suppose that $K_p={\bf Q}_p \oplus {\bf Q}_p.$    Then 
$$K_{m}(d_0,X,t)$$
$$=\sum_{r=0}^{m}{p^{-r^2}(tX^{-1})^r \prod_{i=0}^{r-1}(1-p^{i}X^2) \over \phi_{m-r}(p^{-1})}\zeta_{r}(d_0,tX^{-1}).$$

\noindent
{\rm (3)} Suppose that $K_p$ is ramified over ${\bf Q}_p.$ Then 
$$K_{m}(d_0,X,t)$$
$$=\sum_{r=0}^{m/2}{p^{-4i_pr^2}(tX^{-1})^{(m/2+r)i_p} \prod_{i=0}^{r-1}(1-p^{2i}X^2) \over \phi_{(m-2r)/2}(p^{-2})}\zeta_{2r}((-1)^{m/2-r}d_0,tX^{-1})$$
if $m$ is even, and 
$$K_{m}(d_0,X,t)$$
$$=\sum_{r=0}^{(m-1)/2}{p^{-(2r+1)^2i_p}(tX^{-1})^{((m+1)/2+r)i_p} \prod_{i=0}^{r-1}(1-p^{2i+1}X^2) \over \phi_{(m-2r-1)/2}(p^{-2})}\zeta_{2r+1}((-1)^{(m-2r-1)/2}d_0,tX^{-1})$$
if $m$ is odd.

\end{props}
 
 \begin{proof} 
 The assertions can be proved by using Corollary to Lemma 4.2.2 and  Proposition 4.2.4 (cf. [\cite{I-K3}, Proposition 3.1]). 
\end{proof}

\bigskip

It is well known that $\#({\bf Z}_p^*/N_{K_p/{\bf Q}_p}({\mathcal O}_p^*))=2$ if $K_p/{\bf Q}_p$ is ramified. Hence we can take a complete set ${\mathcal N}_p$ of representatives of ${\bf Z}_p^*/N_{K_p/{\bf Q}_p}({\mathcal O}_p^*)$ so that ${\mathcal N}_p=\{1,\xi_0\}$ with $\chi_{K_p}(\xi_0)=-1.$

 \begin{proof}[{\bf Proof of Theorem 4.3.1.}]
(1)   By Corollary to Proposition 4.3.4 and Proposition 4.3.5, we have 
$$K_m(d_0,X,t)={1 \over \phi_m(-p^{-1})}{L_m(d_0,X,t) \over \prod_{i=1}^m (1+(-1)^ip^{-i}X^{-1}t)},$$
where $L_m(d_0,X,t)$ is a polynomial in $t$ of degree $m.$
Hence 
$$P_m(d_0,X,t)={1 \over \phi_m(-p^{-1})}{L_m(d_0,X,t) \over \prod_{i=1}^m (1+(-1)^ip^{-i}X^{-1}t)\prod_{i=1}^m (1-p^{-2i}X^2t^2) }.$$
We have
$$\widetilde F(B,-X^{-1})=\widetilde F(B,X)$$ 
for any $B \in \widetilde F_p^{(0)}(B,X).$ Hence we have
$$P_m(d_0,-X^{-1},t)=P_m(d_0,X,t),$$
and therefore the denominator of the rational function $P_m(d_0,X,t)$ in $t$ is at most 
$$\prod_{i=1}^m (1+(-1)^ip^{-i}X^{-1}t)\prod_{i=1}^m (1-(-1)^ip^{-i}Xt).$$
Thus
$$P_m(d_0,X,t)={a \over \phi_m(-p^{-1}) \prod_{i=1}^m (1+(-1)^ip^{-i}X^{-1}t)\prod_{i=1}^m (1-(-1)^ip^{-i}Xt)},$$
with some constant $a.$ It is easily seen that we have $a=1.$ This proves the assertion. 

(2) The assertion can be proved by using the same argument as above.

(3) By Corollary to Proposition 4.3.4 and Proposition 4.3.5, we have
$$K_m(d,X,t)$$
$$= {1 \over 2}\left\{ { L^{(0)}(X,t)  \over \prod_{i=1}^{m/2} (1-p^{-2i+1} X^{-1}t)} + {\chi_{K_p}((-1)^{m/2}d_0)L^{(1)}(X,t)   \over \prod_{i=1}^{m/2} (1-p^{-2i} X^{-1}t)} \right\}$$
with some polynomials $L^{(0)}(X,t)$ and $L^{(1)}(X,t)$ in $t$ of degrees at most  $m.$ 
Thus we have
$$P_m(d,X,t)$$
$$= {1 \over 2}\left\{ { L^{(0)}(X,t)  \over \prod_{i=1}^{m/2} (1-p^{-2i+1} X^{-1}t)\prod_{i=1}^{m} (1-p^{-i} Xt)} + {\chi_{K_p}((-1)^{m/2}d_0)L^{(1)}(X,t)   \over \prod_{i=1}^{m/2} (1-p^{-2i} X^{-1}t)\prod_{i=1}^{m} (1-p^{-i} Xt) } \right\}.$$
For $l=0,1$ put 
$$P_{m}^{(l)}(X,t)={1 \over 2} \sum_{d \in {\mathcal N}_p} \chi_{K_p}((-1)^{m/2}d)^lP_m(d,X,t).$$
Then 
$$P_m^{(0)}(X,t)={L^{(0)}(X,t) \over 2\phi_{m/2}(p^{-2})}{ 1 \over \prod_{i=1}^{m/2} (1-p^{-2i+1} X^{-1}t)\prod_{i=1}^{m} (1-p^{-i} Xt) } ,$$
and
$$P_m^{(1)}(X,t)={L^{(1)}(X,t) \over 2\phi_{m/2}(p^{-2})}{ 1 \over \prod_{i=1}^{m/2} (1-p^{-2i} X^{-1}t)\prod_{i=1}^{m} (1-p^{-i} Xt) }.$$
Then by the functional equation of Siegel series we have
$$P_m(d,X^{-1},t)=\chi_{K_p}((-1)^{m/2}d)P_m(d,X,t)$$
for any $d \in {\mathcal N}_p.$ Hence we have
$$P_m^{(0)}(X^{-1},t)=P_m^{(1)}(X,t).$$
Hence the reduced denominator of the rational function $P_m^{(0)}(X,t)$ in $t$ is at most 
$$\prod_{i=1}^{m/2} (1-p^{-2i+1} X^{-1}t)\prod_{i=1}^{m/2} (1-p^{-2i} Xt),$$
and similarly to (1) we have
$$P_m^{(0)}(X,t)= {1 \over 2\phi_{m/2}(p^{-2})\prod_{i=1}^{m/2} (1-p^{-2i+1} X^{-1}t)\prod_{i=1}^{m/2} (1-p^{-2i} Xt)  }.$$
Similarly
$$P_m^{(1)}(X,t)= {1 \over 2\phi_{m/2}(p^{-2})\prod_{i=1}^{m/2} (1-p^{-2i} X^{-1}t)\prod_{i=1}^{m/2} (1-p^{-2i+1} Xt)  }.$$
We have
$$P_m(d_0,X,t)=P_m^{(0)}(X,t)+\chi_{K_p}((-1)^{m/2}d_0)P_m^{(1)}(X,t).$$
This proves the assertion.

\end{proof}

\bigskip

\begin{proof}
[{\bf Proof of Theorem 4.3.2.}]
The assertion can also be proved by using the same argument as above.
\end{proof}

\begin{thms}
Let $d_0 \in {\bf Z}_p^*.$ 

\noindent
{\rm (1)}  Suppose that  $K_p$ is unramified  over ${\bf Q}_p$ or that $K_p={\bf Q}_p \oplus {\bf Q}_p.$ Then
$$\hat P_{m}(d_0,X,t)=P_{m}(d_0,X,t)$$
for any $m >0.$

\noindent
{\rm (2)}  Suppose that  $K_p$ is ramified  over ${\bf Q}_p.$   Then
$$\hat P_{2n+1}(d_0,X,t)=P_{2n+1}(d_0,X,t)$$
and 
$$\hat P_{2n}(d_0,X,t)={ 1 \over 2\phi_{n}(p^{-2})}$$
$$\times \left\{{t^{n i_p } \over \prod_{i=1}^{n} (1-t p^{-2i+1}X^{-1})(1-tp^{-2i}X)} + {\chi_{K_p}((-1)^{n}d_0)  (t \chi_{K_p}(p))^{ni_p} \over \prod_{i=1}^{n} (1-t p^{-2i}\chi_{K_p}(p)X^{-1})(1-tp^{-2i+1}\chi_{K_p}(p)X)}\right \}.$$

\end{thms}
\begin{proof}
The assertion (1) is clear from the definition. We note that 
$P_{m}(d_0,X,t)$ does not depend on the choice of $\pi.$ 
Suppose that $K_p$ is ramified  over ${\bf Q}_p.$ If $m=2n+1,$ then it follows from (3) of Theorem 4.3.2 that we have
$$\lambda^*_{m,p}(\pi^id,X) =\lambda^*_{m,p}(\pi^i,X)$$
for any $d \in {\bf Z}_p^*,$ and in particular we have
$$\lambda^*_{m,p}(p^id_0,X) =\lambda^*_{m,p}(\pi^i,X).$$
This proves the assertion. Suppose that $m=2n.$ Write $\hat P_{2n}(d_0,X,t)$ as
$$\hat P_{2n}(d_0,X,t)=\hat P_{2n}(d_0,X,t)_{even}+\hat P_{2n}(d_0,X,t)_{odd}, $$
where
$$\hat P_{2n}(d_0,X,t)_{even}={1 \over 2}\{\hat P_{2n}(d_0,X,t)+\hat P_{2n}(d_0,X,-t)\},$$
and
$$\hat P_{2n}(d_0,X,t)_{odd}={1 \over 2}\{\hat P_{2n}(d_0,X,t)-\hat P_{2n}(d_0,X,-t)\}.$$
We have
$$\hat P_{2n}(d_0,X,t)_{even}=\sum_{i=0}^{\infty} \lambda^*_{2n,p}(p^{2i}d_0,X,Y)t^{2i}=\sum_{i=0}^{\infty} \lambda^*_{2n,p}(\pi^{2i}d_0,X,Y)t^{2i}$$
and
$$\hat P_{2n}(d_0,X,t)_{odd}=\sum_{i=0}^{\infty} \lambda^*_{2n,p}(p^{2i+1}d_0,X)t^{2i+1}
=\sum_{i=0}^{\infty} \lambda^*_{2n,p}(\pi^{2i+1}d_0 \pi p^{-1},X)t^{2i+1}.$$
Hence we have
$$\hat P_{2n}(d_0,X,t)_{even}={1 \over 2}\{P_{2n}(d_0,X,t)+ P_{2n}(d_0,X,-t)\},$$
and
$$\hat P_{2n}(d_0,X,Y,t)_{odd}={1 \over 2}\{P_{2n}(d_0 \pi p^{-1} ,X,t)-P_{2n}(d_0 \pi p^{-1},X,-t)\},$$
and hence we have
$$\hat P_{2n}(d_0,X,t)=P_{2n}^{(0)}(d_0,X,t)+{1 \over 2} (1+\chi_{K_p}(\pi p^{-1})) \chi_{K_p}((-1)^nd_0) P_{2n}^{(1)}(d_0,X,t)$$
$$+{1 \over 2} (1-\chi_{K_p}(\pi p^{-1})) \chi_{K_p}((-1)^nd_0) P_{2n}^{(1)}(d_0,X,-t).$$
Assume that $\chi_{K_p}(\pi p^{-1})=1.$ Then $\chi(d_0 \pi p^{-1})=\chi(d_0),$ and we have
$$\hat P_{2n}(d_0,X,t)=P_{2n}(d_0,X,t).$$
Suppose that $\chi_{K_p}(\pi p^{-1})=-1.$ Then $\chi(d_0 \pi p^{-1})=-\chi(d_0),$ and we have
$$\hat P_{2n}(d_0,X,t)=P_{2n}^{(0)}(d_0,X,t)+ \chi_{K_p}((-1)^nd_0) P_{2n}^{(1)}(d_0,X,-t)$$
Since $\pi \in N_{K_p/{\bf Q}_p}(K_p^{\times}),$ we have  $\chi_{K_p}(\pi p^{-1})=\chi_{K_p}(p).$ 
 This proves the assertion.
\end{proof}
\begin{xcor}    
Let $m=2n$ be even. Suppose that $K_p$ is ramified over ${\bf Q}_p.$ 
For $l=0,1$ put
$$\hat P_{2n}^{(l)}(X,t)={1 \over 2}\sum_{d \in {\mathcal N}_p} \chi_{K_p}((-1)^nd)^l \hat P_{2n}(d,X,t).$$
Then we have
$$\hat P_{2n}(d,X,t)={1 \over 2} (\hat P_{2n}^{(0)}(X,t)+\chi_{K_p}((-1)^nd)\hat P_{2n}^{(1)}(X,t)),$$
and 
 $$\hat P_{2n}^{(0)}(X,t)=P_{2n}^{(0)}(X,t),$$
and
$$\hat P_{2n}^{(1)}(X,t)=P_{2n}^{(1)}(X,\chi_{K_p}(p) t),$$
\end{xcor}

\bigskip
The following result will be used to prove Theorems 2.3 and 2.4.
\begin{props}
Let $d \in {\bf Z}_p^{\times}.$ 
 Then we have 
$$\lambda_{m,p}^*(d,X)=u_p\lambda_{m,p}(d,X).$$ 
   \end{props}
   
 \begin{proof}
 Let $I$ be the left-hand side of the above equation.  Let 
 $$GL_m({\mathcal O}_p)_1=\{U \in GL_m({\mathcal O}_p) \ | \ \overline{\det U} \det U=1 \}.$$ Then there exists a bijection from  $\widetilde {\rm Her}_{m}(d,{\mathcal O}_p)/GL_m({\mathcal O}_p)_1$ to \\
 $\widetilde {\rm Her}_{m}(dN_{K_p/{\bf Q}_p}({\mathcal O}_p^*),{\mathcal O}_p)/GL_m({\mathcal O}_p).$ Hence 
$$I=\sum_{A \in {\widetilde{\rm Her}}_m(d,{\mathcal O}_p)/GL_{m}({\mathcal O}_p)_1}  {\widetilde F_p^{(0)}(A,X) \over \alpha_p(A)}.$$
Now for $T \in \widetilde {\rm Her}_{m}(d,{\mathcal O}_p),$  let $l$ be the number of $SL_m({\mathcal O}_p)$-equivalence classes in $\widetilde {\rm Her}_{m}(d,{\mathcal O}_p)$ which are $GL_m({\mathcal O}_p)$-equivalent to $T$. Then it can easily be shown that 
$l=l_{p,T}.$
 Hence the assertion holds. 
 \end{proof}

\section{Proof of the main theorem}

{\bf Proof of Theorem 2.3.} For a while put $\lambda_p^*(d)=\lambda_{m,p}^*(d,\alpha_p^{-1}).$
 Then by Theorem 3.4 and Proposition 4.3.7, we have
$$L(s,I_{2n}(f))=\mu_{2n,k,D}\sum_{d} \prod_p (u_p^{-1}\lambda_p^*(d)) d^{-s+k+2n}.$$
Then by (1) and (2) of Theorem 4.3.1, and (1) of Theorem 4.3.6, $\lambda_p^*(d)$  depends only on $p^{{\rm ord}_p(d)}$ if $p \nmid D.$ Hence we write $\lambda_p^*(d)$ as $\widetilde \lambda_p(p^{{\rm ord}_p(d)}).$
 On the other hand, if $p \mid D,$ by (3) of Theorem 4.3. and , (2) of Theorem 4.3.6, 
$\lambda_p^*(d)$ can be expressed as 
$$\lambda_p^*(d)=\lambda_p^{(0)}(d)+ \chi_{K_p}((-1)^nd p^{-{\rm ord}_p(d)})\lambda_p^{(1)}(d),$$
where $\lambda_p^{(l)}(d)$ is a rational number depending only on $p^{{\rm ord}_p(d)}$ for $l=0,1.$ Hence we write $\lambda_p^{(l)}(d)$ as $\widetilde \lambda_p^{(l)}(p^{{\rm ord}_p(d)}).$
Then we have  
$$b_m(f;d)=\sum_{Q \subset Q_D} \prod_{p|d, p \nmid D} \left(u_p^{-1}\widetilde \lambda_{p}(p^{{\rm ord}_p(d)})\prod_{q \in Q} \chi_{K_q}(p^{{\rm ord}_p(d)})\right)$$
$$ \times \prod_{p|d, p \mid D,p \not\in Q} \left(u_p^{-1}\widetilde \lambda_{p}^{(0)}(p^{{\rm ord}_p(d)})\prod_{q \in Q} \chi_{K_q}(p^{{\rm ord}_p(d)})\right)$$
$$ \times \prod_{p|d, p \in Q} \left(u_p^{-1}\widetilde \lambda_{p}^{(1)}(p^{{\rm ord}_p(d)})\prod_{q \in Q, q \not=p} \chi_{K_q}(p^{{\rm ord}_p(d)})\right) 
\prod_{q \in Q} \chi_{K_q}((-1)^n))$$
for a positive integer $d.$ 
We note that for a subset $Q$ of $Q_D$ we have 
$$\chi_Q(m)=\prod_{q \in Q} \chi_{K_q}(m)$$ for an integer $m$ coprime to any $q \in Q,$ and
$$\chi_{Q}'(p)=\chi_{K_p}(p)\prod_{q \in Q, q \not=p} \chi_{K_q}(p)$$
for any $p \in Q.$ 
Hence, by Theorems 4.3.1 and 4.3.6, and Corollary to Theorem 4.3.6, we have 
$$L(s,I_{2n}(f)) =\mu_{2n,k,D}\sum_{Q \subset Q_D} \prod_{p \nmid D} \sum_{i=0}^{\infty} u_p^{-1}\widetilde \lambda_p(p^i)\chi_Q(p^i)p^{(-s+k+2n)i}$$
$$\times \prod_{p \mid D, p \not\in Q} \sum_{i=0}^{\infty} u_p^{-1}\widetilde\lambda_p^{(0)}(p^i)\chi_Q(p^i)p^{(-s+k+2n)i} \chi_Q((-1)^n)$$
$$ \times  \prod_{p \in Q}\sum_{i=0}^{\infty} u_p^{-1}\widetilde\lambda_p^{(1)}(p^i)\left(\prod_{q \in Q, q \not=p} \chi_{K_q}(p^i) \right)p^{(-s+k+2n)i}.$$
$$=\mu_{2n,k,D}\sum_{Q \subset Q_D} \chi_Q((-1)^n)\prod_{p \nmid D} (u_p^{-1}P_{2n,p}(1,\alpha_p^{-1},\chi_Q(p)p^{-s+k+2n}))$$
$$\times \prod_{p \mid D, p \not\in Q} (u_p^{-1}P_{2n,p}^{(0)}(\alpha_p^{-1},\chi_Q(p)p^{-s+k+2n})) \prod_{p \in Q} (u_p^{-1}P_{2n,p}^{(1)}(\alpha_p^{-1},\chi_Q'(p)p^{-s+k+2n})).$$
Now for $l=0,1$ write 
$P_{2n,p}^{(l)}(X,t)$ as 
$$P_{2n,p}^{(l)}(X,t)=t^{ni_p}\widetilde {P}_{2n,p}^{(l)}(X,t),$$
where $i_p=0$ or $1$ according as $4 || D$ and $p=2,$ or not. Notice that $u_p=(1-\chi(p)p^{-1})^{-1}$ if $p \nmid D$ and $u_p=2^{-1}$ if $p | D.$   Hence  we have 
$$L(s,I_{2n}(f)) =\mu_{2n,k,D}\sum_{Q \subset Q_D} \chi_Q((-1)^n) $$
$$ \times \prod_{p \in Q_D'} p^{(-s+k+2n)n} (\prod_{p \in Q_D, p \not\in Q}\chi_Q(p) \prod_{p \in  Q } \chi_Q'(p))^n$$
$$ \times \prod_{p \nmid D} ((1-\chi(p)p^{-1})P_{2n,p}(1,\alpha_p^{-1},\chi_Q(p)p^{-s+k+2n}))$$
$$\times \prod_{p \mid D, p \not\in Q} (2\widetilde P_{2n,p}^{(0)}(\alpha_p^{-1},\chi_Q(p)p^{-s+k+2n})) \prod_{p \in Q} (2\widetilde P_{2n,p}^{(1)}(\alpha_p^{-1},\chi_Q'(p)p^{-s+k+2n})),$$
where $Q_D'=Q_D \backslash \{2 \}$ or $Q_D$ according as $4 || D$ or not.
Note that 
$$2^{2c_D n(-s+k+2n)} \prod_{p \in Q_D'} p^{(-s+k+2n)n}=D^{(-s+k+2n)n},$$
and  $$\prod_{p \in Q_D, p \not\in Q}\chi_Q(p) \prod_{p \in Q } \chi_Q'(p)=1.$$
Thus the assertion follows  from Theorem  4.3.1. 

\bigskip
{\bf Proof of Theorem 2.4.}
The assertion follows directly from Theorems 3.4 and 4.3.2.

\vfil

\noindent
Hidenori KATSURADA \\
Muroran Institute of Technology  27-1 Mizumoto, Muroran, 050-8585, Japan \\
E-mail: hidenori@mmm.muroran-it.ac.jp \\

\end{document}